\documentclass[mn,a4paper,fleqn]{article}
\usepackage{cite,amsmath,amsthm,amssymb,enumerate}
\newtheorem{lemma}{Lemma}
\newtheorem{theorem}{Theorem}
\newtheorem{corollary}{Corollary}
\newtheorem{proposition}{Proposition}
\theoremstyle{definition}
\newtheorem{remark}{Remark}
\newtheorem{definition}{Definition}
\newtheorem{example}{Example}

\usepackage[]{graphicx}

\title{Holomorphic differentials of certain solvable covers of the projective line over a perfect field}

\author{Sophie Marques\footnote{Corresponding author. E-mail:~\textsf{sophie.marques@uct.ac.za}}\; and Kenneth Ward\footnote{E-mail: ~\textsf{kward@american.edu}}}
\begin{document}
\maketitle  

\begin{abstract}
We provide a Boseck-type basis of the space of holomorphic differentials for a large class of solvable covers of the projective line with perfect field of constants of characteristic $p > 0$. Within this class, we also describe the Galois module structure of holomorphic differentials for abelian covers.
\end{abstract}





\allowdisplaybreaks
\section{Introduction}
Let $K=k(x)$ be a rational function field of one variable with perfect field of constants $k$ of characteristic $p > 0$. Our aim in this paper is to construct an explicit basis of holomorphic differentials associated with certain Galois covers of $K$. We let $G$ denote the Galois group of the given cover. Boseck \cite{Bos} first studied this problem in positive characteristic for Artin-Schreier and Kummer extensions via an explicit basis according to Hasse's \emph{standard form}, which always exists over a rational field, but does not exist in general; see the discussion in \S 3.1 and 3.2. Boseck's basis has known generalisations to other settings. For example, for elementary abelian extensions, Garcia \cite{Gar1,Gar2} used bases of this type for elementary abelian extensions to compute Weierstrass points, and Madden \cite{Madden} used these to calculate the rank of the Hasse-Witt matrix. Boseck bases may also be used to understand the $k[G]$-module structure of the $k$-vector space $\Omega_L$ of holomorphic differentials of the field $L$. The description of the $k[G]$-module structure in terms of the group structure of $G$ was first done by Valentini and Madan \cite{VaMa} and later generalised by Rzedowski-Calderon et al. \cite{RzViMa2}. 

Here, we refer to Hasse's standard form over a non-rational field as \emph{global} standard form, as it requires certain conditions in the generating equation for every place in the field (see Definitions \ref{asglobaldef} and \ref{kummerglobaldef}). In this paper, we focus on solvable covers $L/K$ which admit a tower of subextensions, where each step in the tower is Artin-Schreier or Kummer; the existence of a tower of this type for solvable $L/K$ amounts to assuming that the constant field $k$ contains sufficient roots of unity to support the Kummer steps in the tower (for example, it is enough to suppose that $k$ contains the $|G|$th roots of unity). The assumption of a global standard form is employed to create a uniformisation of Boseck-type bases in the Artin-Schreier and Kummer steps of the tower. Our first main result, Theorem \ref{basis}, shows that this permits the identification a $k$-basis of $\Omega_L$. This construction requires a new type of Boseck basis for Kummer extensions, which we give in Lemma \ref{kummer}. Theorem \ref{basis} may be used to produce bases of differentials appearing in other existing literature \cite{Bos,Gar1,Gar2,RzViMa2}. A particular consequence of Theorem \ref{basis} is our second main result, Theorem \ref{modulestructure}, which gives the decomposition of $\Omega_L$ into indecomposable $k[G]$-modules when $L$ is an abelian cover of $K$ possessing global standard form. Definition \ref{Boseckff}, to which we affix the name \emph{global standard function field} and which includes global standard form, is sufficiently weak to apply Theorems \ref{basis} and \ref{modulestructure} to classes of towers, e.g., composites of cyclic generalised Artin-Schreier extensions with Kummer extensions over the rational field $k(x)$ which do not share ramified places, with $k$ algebraically closed. We do not currently know of a test to determine whether a given function field is a global standard function field, but we give a partial answer to this question in \S 5.1 and Remark 6.5 when $K=k(x)$. The difficulty with giving such a description is associated to when a generator can be written in global standard form over an arbitrary function field, which we believe is related to the structure of the class group. Interestingly, in known results (e.g., \cite{KaKo, RzViMa1}), the Boseck invariants and representation of $\Omega_L$ coincide whether or not the base field is rational, which causes us to ask whether the same might also be true for abelian function fields $L/K$ when the base field $K \neq k(x)$.

The plan of the paper is as follows. In \S 2, we define the notion of a global standard function field (Definition \ref{Boseckff}), and we set notation which is held throughout the paper unless specified otherwise. \S 3 gives the uniformised Boseck basis for Kummer and Artin-Schreier extensions, a Riemann-Hurwitz formula for towers, and a discussion of why global standard form does not always exist. \S 4 gives the $k$-basis of $\Omega_L$ for a global standard function field $L$. In \S 5, we give examples of the construction of solvable towers with global standard form, and we employ \cite[Theorem 7]{KaKo} and \cite[Theorem 1]{VaMa} to generalise some cases of known results on cyclic groups using the explicit basis of Theorem \ref{basis} and a \emph{weak} standard form (which is stronger than local standard form, but weaker than global standard form). In \S 6, we give the Galois module structure when $L$ is an abelian cover exhibiting global standard form, following an argument of Rzedowski-Calder{\'{o}}n et al. \cite[Theorem 1]{RzViMa1}. Finally, in \S 7, we give some open questions raised by this work.

\section{Notation and assumptions} 
In this paper, we will work with extensions of the following type, unless specified otherwise. \begin{definition} A function field $L$ with perfect field of constants $k$ of characteristic $p > 0$ will be called a \textbf{\emph{global standard function field}}\label{Boseckff} if it satisfies the following conditions: \begin{itemize}\item There is a rational field $K=k(x)$ over $k$ such that $L/K$ is a finite Galois extension, and that $L/K$ may be expressed as a tower of cyclic Galois extensions \begin{equation} \label{eq1} L =L_r/ L_{r-1} \cdots / L_1/ L_0=K,\end{equation} where for each $i=1,\ldots,r$, the extension $L_i/L_{i-1}$ possesses cyclic Galois group of order $m_i$ (thus $m_i\mid [L:K]$), with either $m_i=p$ or $m_i$ coprime to $p$.  \item The field of constants $k$ contains the $m_i$th roots of unity, for any $i = 1,\ldots,r$ such that $m_i$ is coprime with $p$. (This is done simply to ensure that all cyclic extensions in the tower \eqref{eq1} of degree $m_i$ coprime with $p$ are Kummer.)\item There is a choice of a generator $x$ for the rational field $K$ such that the place at infinity for $x$ is unramified in $L$. (Note that this is always true whenever the constant field $k$ is infinite, or when $k$ is finite and not every rational point ramifies in $L$.) \item For each $i=1,\ldots,r$, it is possible to find a generator in global standard form for $L_i / L_{i-1}$ (see Definitions \ref{asglobaldef} and \ref{kummerglobaldef}). \item If $L_i/L_{i-1}$ is Kummer for a given $i \in \{1,\ldots,r\}$, then there does not exist an integer $d > 1$  such that \begin{equation*}d \mid \gcd(v_{\mathfrak{p}_{i-1}}(c_i), m_i)\end{equation*} for all places $\mathfrak{p}_{i-1}$ of $L_{i-1}$ which ramify in $L_i$. (We will see in \S 3.2 how this cannot occur when $L_{i-1}$ is a rational field, and how this is related to unramified subextensions of $L_i/L_{i-1}$.)\end{itemize}\end{definition}
We will show in the process of proving Theorem \ref{basis} why these are necessary for the construction of the Boseck $k$-basis of $\Omega_L$. Naturally, we would like to know when one may obtain a Galois tower satisfying all these assumptions. In this article, we again emphasise that we do not provide a complete answer to this difficult question; instead, we give some interesting explicit examples where such a tower may be obtained (see \S 5), and leave this as an open problem for future study.

We also establish (\S 3.1 and 3.2) that the existence of a geometric unramified Galois extension of a function field is enough to prove that a global standard form does not always exist; any unramified Kummer or Artin-Schreier extension in global standard form must be a constant extension. In other words, the requirement of existence of global standard form excludes the possibility of unramified geometric steps in the tower. 
It is essentially only the choice of $L$ which is fixed: If one is able to find $k$, $K$, $x$, and $L_i$ ($i=1,\ldots,r$) which matches all the above assumptions, then one obtains a basis of holomorphic differentials.

We denote by $\mathbb{P}_K$ the set of all places in $K$ which ramify in $L$. For simplicity of notation, we henceforth adopt the convention of denoting by $\mathfrak{P}$ a place above $\mathcal{P}$. Also, for each finite place $\mathcal{P} \in \mathbb{P}_K$, we denote by $d_{\mathcal{P}}$ the degree of the place $\mathcal{P}$ and $p_{\mathcal{P}}(x)$ the irreducible polynomial associated with $\mathcal{P}$, which is a prime ideal in $k[x]$. For each $i=1,\ldots,r$, we let $\mathfrak{p}_i = \mathfrak{P}\cap L_i$ and $\mathbb{P}_{L_{i}}$ the set of places of $L_{i-1}$ which ramify in $L_i$. 
\section{Preliminaries}
\subsection{Artin-Schreier extensions}
If $m_i=p$, then the extension $L_{i}/L_{i-1}$ is Artin-Schreier, i.e., there exists a primitive element $y_{i}$, called an \emph{Artin-Schreier generator}, such that \begin{equation*}y_{i}^p - y_{i} = c_{i} \in L_{i-1},\end{equation*} with $c_{i} \neq w^p-w$ for all $w\in L_{i}$ \cite[Theorem 5.8.4]{Vil}. There exists a generator of the Galois group $\text{Gal}(L_i/L_{i-1})\cong \mathbb{Z}/ p \mathbb{Z}$ which acts on the element $y_i$ via $y_{i} \rightarrow y_{i}+1$. As in Definition \ref{Boseckff}, we suppose that places $\mathfrak{p}_{i-1}$ of $L_{i-1}$ above the place $\mathcal{P}_{\infty}$ of $L_{0}=K$ corresponding to the pole of the element $x$ are unramified in $L_i$.

\begin{definition} \label{asglobaldef} We say that an Artin-Schreier generator $y_i$ of $L_i/L_{i-1}$ is in \emph{global standard form} if, for any choice of place $\mathfrak{p}_{i-1}$ of $L_{i-1}$, $v_{\mathcal{P},i}:=v_{\mathfrak{p}_{i-1}} (c_{i})=v_{\mathfrak{p}_{i}} (y_{i}) \geq 0$ if $\mathfrak{p}_{i-1}$ is unramified in $L_i$, and otherwise $v_{\mathcal{P},i}<0$ and $\gcd(v_{\mathcal{P},i},p)=1$. \end{definition}

As mentioned previously, if $L_{i}/L_{i-1}$ were an unramified extension, then the existence of a global standard form for $L_i/L_{i-1}$ would imply that $v_{\mathfrak{p}_{i-1}}(c_i) \geq 0$ at any place $\mathfrak{p}_{i-1}$ of $L_{i-1}$, which would in turn imply that $c_i \in k$. As  $k$ is algebraically closed in $L$, this would then imply that $y_i \in k$, so that the extension $L_i/L_{i-1}$ is trivial (also, by the Riemann-Hurwitz genus formula, we know that unramified geometric Artin-Schreier extensions do not exist over $k(x)$). In particular, as we suppose the existence of a generator in global standard form for each $L_{i}/L_{i-1}$, this implies that none of the Artin-Schreier steps $L_{i}/L_{i-1}$ ($i=1,\ldots,r$) is unramified. For an Artin-Schreier extension over the rational field $k(x)$, the existence of a generator in global standard form for an Artin-Schreier was proven by Hasse \cite{Has}.

We note that any generator in global standard form has the same valuation at a given ramified place, and that one can always find such a generator locally in standard form, i.e., at a single choice of place $\mathfrak{p}_{i-1}$ \cite[Lemma 3.7.7]{Sti}. We call this a \emph{local} standard form. It is only when one has an Artin-Schreier generator in a local standard form that one may give a formula for the ramification index and differential exponent at that place in terms of the generator; this applies at one place but does not imply that this choice of generator gives a global standard form. Given a generator in local standard form, the place $\mathfrak{p}_{i-1}$ is ramified (and hence $e_i=p$) if, and only if $v_{\mathcal{P},i}<0$, and the differential exponent satisfies \begin{equation*}d(\mathfrak{p}_{i}|\mathfrak{p}_{i-1})= (p-1)(1-v_{\mathcal{P},i}).\end{equation*} The ramification filtration $\{G_n(\mathfrak{p}_i)\}_{n=0}^\infty$ has only one jump, occurring at $n=1-v_{\mathcal{P},i}$ \cite[Proposition 3.7.8]{Sti}. We recall here the well known Riemann-Hurwitz formula for geometric Artin Schreier extensions, which becomes important for construction of the Boseck basis, together with the existence of a Artin-Schreier generator in standard form.

\begin{lemma}\cite[Proposition 3.7.8]{Sti} For a geometric Artin-Schreier extension $L/K$, where $K$ is a function field of genus $g_K$, the genus of $L$ is given by \begin{equation*}g_{L} = 1-p + p \cdot g_{K}+\frac{1}{2} \sum_{\mathcal{P}\in \mathbb{P}_K} d_{\mathcal{P}}\cdot (p-1)\cdot J_{\mathcal{P}} ,\end{equation*} where $J_{\mathcal{P}}$ is the jump of the ramification group filtration at the place $\mathcal{P}$, equal to $1-v_{\mathcal{P}}$, and $v_{\mathcal{P}}$ denotes the valuation of the ramified place in standard form.
\end{lemma}

Note that over a rational field, any Artin-Schreier extension has a Artin-Schreier generator in standard form \cite[Example 5.8.8]{Vil}, which permitted Boseck to give an explicit basis for the space of holomorphic differentials of an Artin-Schreier extension of a rational function field.

\begin{lemma}\cite[Satz 15]{Bos}\label{ASbasis}
Let $K=k(x)$, and let $K(y)=k(x,y)$ be a geometric Artin-Schreier extension of $K$ degree $p$, defined by the relation \begin{equation*}y^p-y = \frac{g(x)}{\prod_{i=1}^r p_i(x)^{v_i}},\end{equation*} where for each $i=1,\ldots,r$, $p_i(x)$ and $g(x)$ are relatively prime polynomials in $k[x]$, $d_i$ denotes the degree of the monic irreducible polynomial $p_i(x)$, and 
$v_i \not\equiv 0 \text{ mod } p $. Suppose furthermore that the element $x$ is chosen so that the place at infinity is unramified. The ramified places in $K(y)$ are precisely those of $K$ associated with $p_i(x)$ for each $i=1,\ldots,r$, and these are fully ramified with ramification index $p$. For each $\mu \in \{0, \ldots, p-1 \}$, define $\lambda_i^{\mu}$ and $\rho_i^\mu$ according to the following formula: 
\begin{equation*}p \lambda_i^{\mu}+\rho_i^\mu=(p-1 -\mu) v_i+p-1,\end{equation*} with $0 \leq \rho_i^\mu \leq p -1$. For each such $\mu$, let $g_\mu(x) \in k[x]$ be defined as \begin{equation*}g_\mu (x) = \prod_{i=1}^r (p_i(x))^{\lambda_i^\mu},\end{equation*} and let \begin{equation*}t^\mu = \sum_{i=1}^r d_i \lambda_i^\mu.\end{equation*} Then the set \begin{equation*}\mathfrak{B}_{L} = \left\{ x^\nu [g_\mu (x)]^{-1} y^\mu dx \;\big|\; 0\leq \nu \leq t^\mu-2,  \ 0\leq \mu \leq p-2\right\}\end{equation*} forms a $k$-basis of the space of $\Omega_L$ of holomorphic differentials of $L$.
\end{lemma} 

\subsection{Kummer extensions} 
As for each $i \in \{1,\ldots,r\}$ with $gcd(m_i,p)=1$, the extension $L_i/L_{i-1}$ is Kummer, one may find a primitive element $y_{i}$, called a \emph{Kummer generator}, such that $L_i=L_{i-1}(y_i)$, $y_{i}^{m_i} = c_{i} \in L_{i-1}$, and $c_{i} \neq w^{v}$, for all $v \mid m_i$ and $w \in L_{i}$. For a given primitive $m_i$th root of unity $\zeta$, there exists a generator of the Galois group $\text{Gal}(L_i/L_{i-1}) \cong \mathbb{Z}/m_i \mathbb{Z}$ which acts on the element $y_i$ via $y_i \rightarrow \zeta y_i$. We note that $v_{\mathfrak{p}_{i-1}}(c_i)$ is not divisible by $m_i$ at any place $\mathfrak{p}_{i-1}$ of $L_{i-1}$ ramified in $L_i$. As in Definition \ref{Boseckff}, we suppose that places $\mathfrak{p}_{i-1}$ above the place $\mathcal{P}_{\infty}$ of $L_{0}=K$ corresponding to the pole of the element $x$ are unramified, which is equivalent to $m_i | v_{\mathfrak{p}_{i-1}}(c_i)$.

\begin{definition} \label{kummerglobaldef} We say that a Kummer generator $y_i$ of $L_i/L_{i-1}$ is in \emph{global standard form} if $0\leq v_{\mathfrak{p}_{i-1}}(c_i)<m_i $ at all places $\mathfrak{p}_{i-1}$ of $L_{i-1}$ ramified in $L_i$, and $v_{\mathfrak{p}_{i-1}}(c_i) =0$ at all places $\mathfrak{p}_{i-1}$ of $L_{i-1}$ unramified in $L_i$, with the exception of those places $\mathfrak{p}_{i-1}$ of $L_{i-1}$ above the place $\mathcal{P}_{\infty}$ of $L_{0}=K$ corresponding to the pole of the element $x$ for which we suppose that $v_{\mathfrak{p}_{i-1}}(c_i)\leq 0$. \end{definition}

The existence of a generator in global standard form for a Kummer extension over a rational field has been proven by Hasse \cite{Has}. Over an arbitrary function field $K$, we know that a generator in global standard form does not always exist. As evidence, we use the case where $K$ possesses au unramified geometric extension (this implies, in particular, that $K$ is not rational). For such a $K$ and such an extension, let $y$ be a generator such that $y^n = c \in K$. If $y$ were in global standard form, then $v_{\mathcal{P}}(c) =0$  at all places $\mathcal{P}$ that are not above $\mathcal{P}_{\infty}$ and $v_{\mathcal{P}}(c)\leq 0$ at the places above infinity, which is impossible unless the extension is constant. As a global standard form does exist over a rational field, this also implies that unramified geometric Kummer extensions do not exist over a rational field, a fact which also follows by Riemann-Hurwitz. 

In order to prove Theorem \ref{basis}, it is necessary to assume that there exists no integer $d > 1$  such that \begin{equation*}d \mid \gcd(v_{\mathfrak{p}_{i-1}}(c_i), m_i)\end{equation*} for all places $\mathfrak{p}_{i-1}$ of $L_{i-1}$ which ramify in $L_i$. This is linked to the existence of unramified geometric subextensions in $L_i/L_{i-1}$, for which we know that global standard form does not exist. To see this, suppose that $v_{{p}_{i-1}}(c_i)=l_{{p}_{i-1}}$ for any such ramified place and denote $d=\gcd(l_{{p}_{i-1}},m_i)$ this common factor. With $u= y^{m_i/d}$, the Kummer subextension $L_{i-1}(u)/L_{i-1}$ has $u^d = c_i$, and by assumption, $d \mid v_{{p}_{i-1}}(c_i)$, for any ramified places $\mathfrak{p}_{i-1}$ in $L_i / L_{i-1}$. As a consequence of this, $L_{i-1}(u)/L_{i-1}$ is unramified. Boseck did not need to assume this in creating an explicit basis for Kummer extensions over the rational field, as unramified extensions of the rational field simply do not exist. The same is true in the work of Valentini and Madan \cite{VaMa}. 

It is known that for a fixed choice of place $\mathfrak{p}_{i-1}$ of $L_{i-1}$, $y_i$ may be chosen in \emph{local} standard form at $\mathfrak{p}_{i-1}$ \cite[Theorem 5.8.12]{Vil}, so that $0= v_{\mathfrak{p}_{i-1}}(c_{i})$ if $\mathfrak{p}_{i-1}$ is unramified and $v_{\mathfrak{p}_{i-1}}(c_{i})>0$ if $\mathfrak{p}_{i-1}$ ramified in $L_i$, where the valuation $v_{\mathfrak{p}_{i-1}}(c_{i})$ of $c_i$ is viewed in the field $L_{i-1}$ via $y_i^{m_i} = c_i$. For each place $\mathfrak{p}_i$ of $L_i$ ramified above $L_{i-1}$, the ramification index in $L_i/L_{i-1}$ satisfies \begin{equation*}e(\mathfrak{p}_{i}|\mathfrak{p}_{i-1})=\frac{m_i}{\gcd(m_i,v_{\mathfrak{p}_{i-1}}(c_{i}))},\end{equation*} where the valuation $v_{\mathfrak{p}_{i-1}}$ is in local standard form, and the differential exponent is equal to \begin{equation*}d(\mathfrak{p}_{i}|\mathfrak{p}_{i-1})=\frac{m_i}{\gcd(m_i,v_{\mathfrak{p}_{i-1}}(c_{i}))}-1.\end{equation*} We note furthermore that $v_{\mathcal{P},i} := v_{\mathfrak{p}_i} (y_i) = e(\mathfrak{p}_{i}|\mathfrak{p}_{i-1})v_{\mathfrak{p}_{i-1}}(c_{i})/ m_i$ is coprime with $n$, for any place $\mathfrak{p}_i$ of $L_i$ above a ramified place $\mathfrak{p}_{i-1}$ of $L_{i-1}$.  
A Kummer extension is of degree coprime to $p$, thus tamely ramified, and as a consequence the ramification filtration $\{G_n(\mathfrak{p}_i)\}_{n=0}^\infty$ at $\mathfrak{p}_i$ has only one jump, which occurs at $n=1$ \cite[Proposition 3.7.3]{Sti}. We recall again the well-known genus formula for geometric Kummer extensions (ibid.).

\begin{lemma} \label{boseck} For a geometric Kummer extension $L/K$ of degree $n$, where $K$ is a function field of genus $g_K$, the genus of $L$ is given by \begin{equation*}g_{L} = 1-n + n \cdot g_{K}+\frac{1}{2} \sum_{\mathcal{P}\in \mathbb{P}_K} (e(\mathfrak{P}|\mathcal{P})-1) \cdot \frac{n}{e(\mathfrak{P}|\mathcal{P})}\cdot d_{\mathcal{P}}.\end{equation*} \end{lemma}

Unlike in Artin-Schreier extensions, the term $n/e(\mathfrak{P}|\mathcal{P})$ appears in the Riemann-Hurwitz formula, as partial ramification is possible in Kummer extensions. In order to construct a $k$-basis of $\Omega_L$ using local ramification data, we must transform Boseck's formulae for Kummer and Artin-Schreier extensions to obtain a type of basis which is consistent across such extensions. The primary obstruction to a direct application of the Boseck basis is the sign difference in the power of the generator in the Kummer versus the Artin-Schreier case: For Kummer extensions, the power of a generator appearing in the Boseck basis is negative, whereas it is positive for Artin-Schreier extensions. We therefore construct an alternate form of the Boseck basis for Kummer extensions in the following lemma.
\begin{lemma} \label{kummer}
Let $K=k(x)$, and let $K(y)=k(x,y)$ be a geometric Kummer extension of $K$ degree $n$, defined by the relation \begin{equation*}y^n = f(x) \in k[x]\end{equation*} in global standard form with \begin{equation*}f(x) =\alpha \cdot \prod_{i=1}^r (p_i(x))^{v_i},\end{equation*} where for each $i=1,\ldots,r$, $d_i$ denotes the degree of the monic irreducible polynomial $p_i(x)$ and $ 0 < v_i $. Suppose that the place of $K$ at infinity, corresponding to the pole of $x$, is unramified in $K(y)$ (which is equivalent to requiring $n \mid \deg f(x)$). Let $ v=\sum_{i=1}^r v_i d_i$ denote the degree of $f(x)$, and for each $i=1,\ldots,r$, let $m_i = \frac{e_i v_i}{n}$, where $e_i$ is the ramification index for $p_i(x)$ in $K(y)$. For each $\mu \in \{1, \ldots , n-1 \}$, we define $\lambda_i^{\mu}$ and $\rho_i^\mu$ according to the following formula: \begin{equation*}e_i \lambda_i^{\mu}+\rho_i^\mu= \mu m_i+e_i-1,\end{equation*} with $0 \leq \rho_i^\mu \leq e_i -1$. For each such $\mu$, let $g_\mu(x) \in k[x]$ be defined as \begin{equation*}g_\mu (x) = \prod_{i=1}^r (p_i(x))^{\lambda_i^\mu},\end{equation*} and let \begin{equation*}t^\mu = \sum_{i=1}^r \frac{d_i}{e_i} (e_i - 1 -\rho_i^\mu).\end{equation*} Then the set \begin{equation*}\mathfrak{B}_{K(y)} = \left\{ x^\nu [g_\mu (x)]^{-1} y^\mu dx \;\big|\; 0\leq \nu \leq t^\mu-2,  \ 1\leq \mu \leq n-1\right\}\end{equation*} forms a $k$-basis of the space of $\Omega_{K(y)}$ of holomorphic differentials of $K(y)$.
\end{lemma} 
\begin{proof} The argument follows similarly to the proof of \cite[Satz 16]{Bos}. For each $i=1,\ldots,r$, $m_i$ is equal to the valuation of $y$ at a ramified place of $K(y)/K$. We thus find that the divisor of the differential $[g_\mu(x)]^{-1} y^\mu dx$ in $K(y)$ is equal to  \begin{align*} ([g_\mu(x)]^{-1} y^\mu dx)_{K(y)} &= \prod_{i=1}^r \mathfrak{P}_i^{e_i-1 - e_i \lambda_i^\mu + \mu m_i} \cdot (\text{Con}_{K/K(y)}(\mathcal{P}_\infty))^{\sum_{i=1}^r d_i \lambda_i^\mu - \mu \left(\frac{1}{n}\sum_{i=1}^r v_i d_i\right) -2} \\&= \prod_{i=1}^r \mathfrak{P}_i^{e_i-1 - e_i \lambda_i^\mu + \mu m_i} \cdot (\text{Con}_{K/K(y)}(\mathcal{P}_\infty))^{\frac{1}{n} \sum_{i=1}^r d_i\frac{n}{e_i} \left( e_i \lambda_i^\mu - \mu  m_i \right) -2} \\& =\prod_{i=1}^r \mathfrak{P}_i^{\rho_i^\mu} \cdot (\text{Con}_{K/K(y)}(\mathcal{P}_\infty))^{\sum_{i=1}^r d_i(e_i-1-\rho_i^\mu)/e_i - 2},\end{align*} where $\text{Con}_{K/K(y)}$ denotes the conorm of ideals of $K$ into $K(y)$. By definition of $t^\mu$, it follows that the differential $[g_\mu(x)]^{-1} y^\mu dx$ is holomorphic provided that $t^\mu \geq 2$. 
By the requirement $\mu \in \{1,\ldots,n-1\}$, it follows that $t^\mu \geq 1$. To see this, notice that by definition, $\rho_i^{\mu} \leq e_i-1$, and thus $t^\mu$ is always nonnegative (and an integer; see \cite[Satz 16]{Bos}). If $t^\mu < 1$, then $t^\mu = 0$, which then implies that $\rho_i^{\mu}=e_i-1$ for all $i =1,\ldots,r$. It follows for each $i=1,\ldots,r$ that \begin{equation*}e_i \lambda_i^\mu = \mu m_i = \frac{\mu e_i v_i }{n}.\end{equation*} In particular, we obtain that $n \mid \mu v_i$. As $\mu < n$, it follows that there is a factor $d$ of $n$ ($d > 1$) which divides $v_i$, for all $i = 1,\ldots,r$. We thus obtain $y^n = z^d$ with $z\in K$. Let $u=(y^{n/d}/z) \in K$, then $u^d=1$. 
As $k$ contains the $dth$ roots of unity, this contradicts that the degree of $K(y)/K$ is equal to $n$. Also, by definition of $\mathfrak{B}_L$, there exist no holomorphic differentials with $t^\mu = 1$. Therefore, the number of such differentials which are holomorphic is equal to \begin{equation}\label{eq2} \sum_{\mu=1}^{n-1} (t^\mu - 1),\end{equation} as $t^\mu = 0$ cannot occur as mentioned previously and $t^\mu = 1$ does not contribute to this sum. Therefore, by Riemann-Hurwitz \cite[Corollary 9.4.3]{Vil}, the quantity \eqref{eq2} is equal to \begin{align*} \sum_{\mu=1}^{n-1} (t^\mu - 1) &= \sum_{\mu=1}^{n-1} \left[\left(\sum_{i=1}^r \frac{d_i}{e_i}(e_i-1-\rho_i^\mu )\right) -1\right]\\& = \sum_{i=1}^r d_i \frac{n}{e_i}\left(\frac{e_i-1}{2}\right) - (n-1)\\& = -1 + n(g_{K}-1) + \frac{1}{2} \sum_{i=1}^r d_i \frac{n}{e_i} (e_i-1) \\& = g_{K(y)}, \end{align*} where the second equality above is justified by \begin{align*} \sum_{\mu=1}^{n-1} (e_i - 1 - \rho_i^\mu) & =  \sum_{\mu=0}^{n-1} (e_i - 1 - \rho_i^\mu) \\& = \sum_{k=1}^{n/e_i}\sum_{\mu=1+(k-1)e_i}^{ke_i-1}  (e_i - 1 - \rho_i^\mu) \\& = \sum_{k=1}^{n/e_i}\sum_{\mu=1+(k-1)e_i}^{ke_i-1}  \rho_i^\mu \\&= \sum_{k=1}^{n/e_i} \frac{e_i (e_i-1)}{2} \\& = n\frac{(e_i-1)}{2}.\end{align*} This follows from the identity \begin{equation*}\sum_{\mu=1+(k-1)e_i}^{ke_i-1}  (e_i - 1 - \rho_i^\mu)  = \sum_{\mu=1+(k-1)e_i}^{ke_i-1} \rho_i^\mu= \frac{e_i (e_i-1)}{2},\end{equation*} which holds as $\gcd(m_i, n)=1$, so that the quantities \begin{equation*}\mu m_i+e_i-1 \;\;\;\;\;\; (\mu = 1+(k-1)e_i,\ldots,ke_i-1)\end{equation*} form a complete set of residues modulo $e_i$. It follows that the elements of $\mathfrak{B}_{K(y)}$ form a $k$-basis of $\Omega_{K(y)}$.
\end{proof}

As in Lemma 3.3, the $k$-basis of $\Omega_L$ in the case that $L/K$ is an Artin-Schreier extension and $K=k(x)$ is the rational function field also consists of elements of the form $x^\nu [g_\mu (x)]^{-1} y^\mu dx$. This gives a unified expression for the Boseck basis for both Artin-Schreier and Kummer extensions, which in the sequel will allow us to generate the basis for the mixed solvable tower \eqref{eq1}.

\subsection{The Riemann-Hurwitz formula for towers}
By previous arguments, for both Artin-Schreier and Kummer extensions, the differential exponent at any place $\mathfrak{p}_{i-1} \in \mathbb{P}_{L_i}$ is given by \begin{equation*}d(\mathfrak{p}_{i}|\mathfrak{p}_{i-1})= (e(\mathfrak{p}_{i}|\mathfrak{p}_{i-1})-1)J_{\mathcal{P},i},\end{equation*} where $J_{\mathcal{P},i}$ is the unique jump of the ramification filtration for $\mathfrak{p}_i$ in $L_i/L_{i-1}$. By Riemann-Hurwitz, we thus obtain the following genus formula in either situation.

\begin{lemma} For the extension $L_i/L_{i-1}$, the genus formula is given by \begin{equation*}g_{L_i} = 1-m_i + m_i \cdot g_{L_{i-1}}+\frac{1}{2} \sum_{\mathfrak{p}_{i-1}\in \mathbb{P}_{L_i}}\frac{m_i}{e(\mathfrak{p}_{i}|\mathfrak{p}_{i-1})} \cdot (e(\mathfrak{p}_{i}|\mathfrak{p}_{i-1})-1)\cdot J_{\mathcal{P},i}\cdot d_{\mathfrak{p}_{i-1}}.\end{equation*} \end{lemma} The genus formula for $L_i/L_{i-1}$ ($i=1,\ldots,r$) allows us to obtain a concise genus formula for the Galois tower $L/K$. The following result accomplishes exactly this. \begin{lemma}\label{RH} \begin{enumerate}[(i)] 
\item The differential exponent $d(\mathfrak{P}|\mathcal{P})$ of $\mathfrak{P}|\mathcal{P}$ in $L/K$ is given by \begin{align} \notag \begin{array}{ll} d(\mathfrak{P}|\mathcal{P}) & = \sum_{i\in R_\mathcal{P}} e(\mathfrak{P}|\mathfrak{p}_{i}) \cdot (e(\mathfrak{p}_{i}|\mathfrak{p}_{i-1})-1)\cdot J_{\mathcal{P},i},\end{array} \end{align} where $R_\mathcal{P} \subset \{0,1,\ldots,r-1\}$ denotes the set of indices such that the place $\mathfrak{p}_{i-1}$ is ramified in $L_i/ L_{i-1}$. 
\item The Riemann-Hurwitz formula for $L/K$ may be written as \begin{equation*}\begin{array}{lll}g_L&=1 - [L:K] + \frac{1}{2} \sum_{\mathcal{P}\in \mathbb{P}_K} \frac{[L:K]}{e(\mathfrak{P}|\mathcal{P})}  \cdot d_{\mathcal{P}} \cdot\left[ \sum_{i\in R_\mathcal{P}}  e(\mathfrak{P}|\mathfrak{p}_{i}) \cdot (e(\mathfrak{p}_{i}|\mathfrak{p}_{i-1})-1)\cdot J_{\mathcal{P},i}\right]\end{array}\end{equation*} \end{enumerate} \end{lemma}
\begin{proof} For all $i =1,\ldots,r$, we have the ramification formula \begin{equation*}e(\mathfrak{p}_{i}| \mathcal{P} ) = e(\mathfrak{p}_{i}|\mathfrak{p}_{i-1}) e(\mathfrak{p}_{i-1} |  \mathcal{P})\end{equation*} and differential exponent \begin{equation*}d( \mathfrak{p}_{i} |  \mathcal{P}) = e ( \mathfrak{p}_{i} | \mathfrak{p}_{i-1}) d( \mathfrak{p}_{i-1} |  \mathcal{P}) + d( \mathfrak{p}_{i} | \mathfrak{p}_{i-1}).\end{equation*} From previous observations, have $d( \mathfrak{p}_{i} | \mathfrak{p}_{i-1})= (e(\mathfrak{p}_{i}|\mathfrak{p}_{i-1})-1)\cdot J_{\mathcal{P},i}$ for each $i=1,\ldots,r$. Thus, the formula for the differential exponent of $\mathfrak{P}|\mathcal{P}$ may be expressed as \begin{equation*} d(\mathfrak{P}|\mathcal{P}) = \sum_{i\in R_\mathcal{P}} e(\mathfrak{P}|\mathfrak{p}_{i}) (e(\mathfrak{p}_{i}|\mathfrak{p}_{i-1})-1)\cdot J_{\mathcal{P},i},\end{equation*} proving (i). For (ii), by definition, the different $\mathcal{D}_{L/K}$ of $L$ over $K$ is equal to \begin{equation*}\mathcal{D}_{L/K}= \prod_{\mathcal{P}\in \mathbb{P}_K} \mathcal{P}r_{\mathcal{P}}^{d(\mathfrak{P}| \mathcal{P})},\end{equation*} where $\mathcal{P}r_{\mathcal{P}}$ denotes the product of all places of $L$ above $\mathcal{P}$. As $L/K$ is Galois, the inertia degree and ramification index at a place $\mathfrak{P}$ of $L$ above $\mathcal{P} \in \mathbb{P}_K$ is independent of the choice of $\mathfrak{P}$. Hence, the product of the inertia degree of $\mathfrak{P}|\mathcal{P}$ with the number of places of $L$ above $\mathcal{P}$ is equal to $[L:K]/e(\mathfrak{P}|\mathcal{P})$ \cite[Corollary 5.2.23]{Vil}. Furthermore, the differential exponent in $L$ at a place $\mathfrak{P}|\mathcal{P}$ is also independent of the choice of $\mathfrak{P}$. With the help of \cite[Corollary 3.1.14]{Sti}, the Riemann-Hurwitz formula for $L/K$ may thus be written \cite[Corollary 9.4.3]{Vil} as \begin{equation*}g_L=1 - [L:K] + \frac{1}{2} \sum_{\mathcal{P}\in \mathbb{P}_K} \frac{[L:K]}{e(\mathfrak{P}|\mathcal{P})}  \cdot d_{\mathcal{P}} \cdot\left[ \sum_{i\in R_\mathcal{P}}  e(\mathfrak{P}|\mathfrak{p}_{i}) (e(\mathfrak{p}_{i}|\mathfrak{p}_{i-1})-1)\cdot J_{\mathcal{P},i}\right],\end{equation*} as desired.\end{proof} 

This is a convenient formula for the genus, as it is expressed only in terms of the valuations of global standard form generators and ramification data for the tower. 

\begin{remark} Lemma \ref{RH} remains valid if $L/K$ is separable but not Galois, provided ramification indices, inertia degrees, and differential exponents are equal for all places of $L$ above a given place of $K$, for all places of $K$ which ramify in $L$. \end{remark}

\section{Basis of holomorphic differentials} 
In this section, we provide an explicit description of the $k$-basis of $\Omega_L$, which is our first main result. This is done strictly in terms of the ramification data and valuations of global standard form generators of the tower $L/K$. Our construction additionally requires the modified Boseck basis for Kummer extensions introduced in Lemma \ref{boseck}, which allows the tower to consist of steps of both Artin-Schreier and Kummer extensions.

\begin{theorem} \label{basis}
Let $L$ be a global standard function field (Definition \ref{Boseckff}). For each $i=1,\ldots,r$, we suppose that each $y_i$ is either a Kummer or Artin-Schreier generator in global standard form. Given a place $\mathcal{P} \in \mathbb{P}_K$, let $R_{ \mathcal{P}}= \{i \in \{1, \ldots , r\}, \mathfrak{p}_{i-1}\in \mathbb{P}_{L_i} \}$ the set of indices $i=1,\ldots,r$ such that $\mathfrak{p}_{i-1}$ ramifies in $L_i/ L_{i-1}$, let \begin{equation*}R_{p,\mathcal{P}}= \{ i\in \{ 1, \ldots , r\}| m_i=p, \mathfrak{p}_{i-1} \in \mathbb{P}_{L_i}\}\end{equation*} denote the $p$-subset of $R_{ \mathcal{P}}$, i.e., those indices such that $m_i=p$, and let \begin{equation*}R_{o, \mathcal{P}}= \{ i\in \{ 1, \ldots , r\}| m_i \neq p, \mathfrak{p}_{i-1} \in \mathbb{P}_{L_i}\}\end{equation*} denote the prime-to-$p$ subset of $R_{ \mathcal{P}}$, i.e., those such that $m_i\neq p$. We denote $e_{\mathcal{P}}:= e( \mathfrak{P}|\mathcal{P})$ and $v_{\mathcal{P},i} := v_{\mathfrak{p}_{i}} (y_i)$, where the valuation of $y_i$ is viewed as existing in $L_{i}$. For each $i=1,\ldots,r$ and $\mu_i$, let $\Delta_{\mathfrak{p}_i}^{\mu_i}$ be defined according to the following formula: \begin{equation} \label{delta} \Delta_{\mathfrak{p}_i}^{\mu_i} =  \left\{ \begin{array}{lll}  (p-1 - \mu_i )\cdot (-v_{\mathcal{P},i}) +(p-1) & if \ i \in R_{p,\mathcal{P}}  \\ \mu_i  v_{\mathcal{P},i} +(e(\mathfrak{p}_i|\mathfrak{p}_{i-1})-1) & if \ i \in R_{o,\mathcal{P}}\\ 0 & otherwise. \end{array} \right. \end{equation}

Set $\mu = (\mu_1, \ldots , \mu_r)$. Let $\lambda_{\mathcal{P}}^{\mu} $ and $\rho_{\mathcal{P}}^{\mu}$ be defined by the equation \begin{equation} \label{modular} e_{\mathcal{P}} \lambda_{\mathcal{P}}^{\mu} +\rho_{\mathcal{P}}^{\mu}= \sum_{i=1}^r e(\mathfrak{P}| \mathfrak{p}_i) \Delta_{\mathfrak{p}_i}^{\mu_i},\;\;\;\;\;\;0\leq \rho_{\mathcal{P}}^{\mu}\leq e_{\mathcal{P}}-1.\end{equation} Also, let \begin{equation*}g_{\mu} (x) = \prod_{\mathcal{P}\in \mathbb{P}_K} (p_{\mathcal{P}}(x))^{\lambda_{\mathcal{P}}^{\mu}}.\end{equation*} Let ${y}^{{\mu}}= \prod_{j=1}^r y_{j}^{\mu_j}$ and \begin{equation*}t^{\mu} = \sum_{\mathcal{P} \in \mathbb{P}_K} d_{\mathcal{P}} \left(\lambda_{\mathcal{P}}^{\mu} -  \sum_{i\in R_{o,\mathcal{P}}} \frac{ e ( \mathfrak{P} | \mathfrak{p}_{i})}{e_{\mathcal{P}}}  v_{\mathcal{P},i}\mu_i\right),\end{equation*} where $d_\mathcal{P}$ denotes the degree of the place $\mathcal{P} \in \mathbb{P}_K$ (\S 2). Define $\Gamma:= \prod_{i=1}^r \{0 , \ldots , m_i-1\} - \mu^0$, where $\mu^0= (\mu_1^0 , \ldots , \mu_r^0 )$ with $\mu_i^0 = 0 $ if $m_i \neq p$, and $\mu^0_i=m_i-1=p-1$ otherwise. Then the set
\begin{equation*}\mathfrak{B}_L:=\left\{ x^\nu [g_{\mu}(x)]^{-1} y^{\mu} dx \;\big|\; 0\leq \nu \leq t^\mu-2,  \ \mu = (\mu_1, \ldots, \mu_r ) \in \Gamma \right\}\end{equation*} forms a $k$-basis of $\Omega_L$.
\end{theorem}
\begin{proof}
 The divisor of ${y}^{{\mu}}$ in $L$ is given by \begin{equation*}({y}^{{\mu}})_L =\mathfrak{A}_{{y}}^{{\mu}}\cdot \prod_{\mathcal{P} \in \mathbb{P}_K} \mathcal{P}r^{\sum_{i\in R_{\mathcal{P}}} \mu_i e(\mathfrak{P} | \mathfrak{p}_{i} )v_{\mathcal{P},i}},\end{equation*} for some integral divisor $\mathfrak{A}_{{y}}^{{\mu}}$ of $L$, where $\mathcal{P}r$ is the product of the places above $\mathcal{P}$ in $L$. It follows that the divisor in $L$ of the differential ${{y}}^{{\mu}} dx$ is given by \begin{equation*}({y}^{{\mu}}dx)_L=  \mathfrak{A}_{{y}}^{{\mu}} \cdot \prod_{\mathcal{P} \in \mathbb{P}_K} \mathcal{P}r^{\sum_{i\in R_{\mathcal{P}}} e(\mathfrak{P}| \mathfrak{p}_i)\Delta_{\mathfrak{p}_i}^{\mu_i}} (\text{Con}_{K/L} ( \mathcal{P}_\infty))^{-2}.\end{equation*} As the quantities $\lambda_{\mathcal{P}} ^{{\mu}}$ and $\rho_{\mathcal{P}}^\mu$ are defined according to \eqref{modular}, multiplication of ${y}^{{\mu}}dx$ by $[g_{{\mu}}(x)]^{-1}= \prod_{\mathcal{P}\in \mathbb{P}_K} (p_{\mathcal{P}}(x))^{-\lambda_{\mathcal{P}}^{{\mu}}}$ yields the following divisor in $L$: \begin{equation*}([g_{{\mu}}(x)]^{-1} {y}^{{\mu}} dx)_L = \mathfrak{A}_{{y}}^{{\mu}} \cdot \prod_{\mathcal{P}\in \mathbb{P}_K} \mathcal{P}r^{ \rho_{\mathcal{P}}^{{\mu}} }  \cdot (\text{Con}_{K/L} ( \mathcal{P}_\infty))^{  \sum_{\mathcal{P} \in \mathbb{P}_K} d_{\mathcal{P}} \left(\lambda_{\mathcal{P}}^{\mu} -  \sum_{i\in R_{o,\mathcal{P}}} \frac{ e ( \mathfrak{P} | \mathfrak{p}_{i})}{e_{\mathcal{P}}}  v_{\mathcal{P},i}\mu_i\right) -2}.\end{equation*} Thus, the differential $[g_{{\mu}}(x)]^{-1} {y}^{{\mu}} dx$ is holomorphic if, and only if, \begin{equation*}t^{{\mu}} :=  \sum_{\mathcal{P} \in \mathbb{P}_K} d_{\mathcal{P}} \left(\lambda_{\mathcal{P}}^{\mu} -  \sum_{i\in R_{o,\mathcal{P}}} \frac{ e ( \mathfrak{P} | \mathfrak{p}_{i})}{e_{\mathcal{P}}}  v_{\mathcal{P},i}\mu_i\right) \geq 2.\end{equation*} 
 Therefore, as the set $\{x^\nu y^\mu \;|\; \ 0\leq \nu \leq t^\mu-2,\ \mu \in \Gamma \}$ is linearly independent over $k$, the $k$-linearly independent set \begin{equation*}\mathfrak{B}_L = \left\{ x^\nu [g_{\mu}(x)]^{-1} y^{\mu} dx \;\big|\; 0\leq \nu \leq t^\mu-2, \ \mu = (\mu_1, \ldots, \mu_r ) \in \Gamma \right\}\end{equation*} consists solely of holomorphic differentials. Furthermore, we have that $t^{\mu}\geq 1$ for all $\mu \in \Gamma$: By construction, the integer \begin{equation*}\lambda_{\mathcal{P}}^{\mu} -  \sum_{i\in R_{o,\mathcal{P}}} \frac{ e ( \mathfrak{P} | \mathfrak{p}_{i})}{e_{\mathcal{P}}}  v_{\mathcal{P},i}\mu_i = \frac{1}{e_{\mathcal{P}}} \left( e_{\mathcal{P}}-1 - \rho_{\mathcal{P}}^{\mu} + \sum_{i \in R_{p, \mathcal{P}}} (p-1- \mu_i)\cdot (-v_{\mathcal{P}, i}) \right)\geq 0.\end{equation*} Also by construction, $e_{\mathcal{P}}-1 - \rho_{\mathcal{P}}^{\mu} \geq 0$ and $ (p-1-\mu_i) \cdot (-v_{\mathcal{P}, i}) \geq 0$. It follows that $ t^\mu= 0$ if and only if $(p-1- \mu_i)\cdot (-v_{\mathcal{P}, i})=0$, i.e., $\mu_i=p-1$, for all $i \in R_{p, \mathcal{P}}$, and $ \rho_{\mathcal{P}}^{\mu}=e_{\mathcal{P}}-1 $, for any $\mathcal{P} \in \mathbb{P}_K$. For any elements \begin{equation*}s= (s_i)_{i\in R_{0, \mathcal{P}}} \in S:=\prod_{i\in R_{0, \mathcal{P}}} \{0, \ldots , m_i /e (\mathfrak{p}_i|\mathfrak{p}_{i-1})-1\},\end{equation*} we consider the set \begin{equation*}\Gamma_{s, \mathcal{P}} :=\prod_{i \in R_{p, \mathcal{P}}}\{ 0 , \ldots , m_i-1\} \prod_{i \in R_{0, \mathcal{P}}}\{ s_ie (\mathfrak{p}_i|\mathfrak{p}_{i-1}), \ldots , (s_i+1)e (\mathfrak{p}_i|\mathfrak{p}_{i-1})-1\}.\end{equation*} As the elements $v_{\mathcal{P}, i}= e(\mathfrak{p}_i|\mathfrak{p}_{i-1})v_{\mathfrak{p}_{i-1}}(c_i)/m_i$ are coprime with $m_i$ for each $i = 1,\ldots,r$, $v_{\mathcal{P},i}> 0$ for each $i = 1,\ldots,r$ such that $L_i/L_{i-1}$ is Kummer, and $-v_{\mathcal{P}, i}>0$ for each $i = 1,\ldots,r$ such that $L_i/L_{i-1}$ is Artin-Schreier, at any ramified place $\mathfrak{p}_{i-1}$ of $L_i/L_{i-1}$, it follows that for each $s\in S$, the elements $\rho_\mathcal{P}^\mu$ form a complete set of residues modulo $e_{\mathcal{P}}$ as $\mu$ runs through all possible values in the set $\Gamma_{s,\mathcal{P}}$. 
To see this, we have the identity \begin{equation*}\rho_{\mathcal{P}}^{\mu}= e_{\mathcal{P}} \left\langle\frac{\sum_{i=1}^r e(\mathfrak{P}| \mathfrak{p}_i)\Delta_{\mathfrak{p}_i}^{\mu_i}}{e_{\mathcal{P}}}\right\rangle,\end{equation*}
where for an element $x \in \mathbb{R}$, $\langle x \rangle$ denotes the fractional part of $x$. Also, by construction, the set \begin{equation*}\left\{\sum_{i=1}^r e(\mathfrak{P}| \mathfrak{p}_i)\Delta_{\mathfrak{p}_i}^{\mu_i}\right\}_{\mu \in \Gamma_{s,\mathcal{P}}}\end{equation*} forms a complete set of residues modulo $e_{\mathcal{P}}$. Therefore, the remainder $\rho_{\mathcal{P}}^\mu$ assumes the value $e_{\mathcal{P}}-1$ exactly $|S|= \prod_{i\in R_{0, \mathcal{P}}} m_i /e (\mathfrak{p}_i|\mathfrak{p}_{i-1})$ times, and this occurs precisely when the values of $\mu_i$ are multiples of $e (\mathfrak{p}_i|\mathfrak{p}_{i-1})$, for all $i \in R_{0, \mathcal{P}}$. The number of instances where this occurs is equal to $|S|$, which subsumes all possible values of $\mu$ for which $e_{\mathcal{P}}-1= \rho_{\mathcal{P}}^{\mu}$. 

As argued in the proof of Lemma 3.6, in order to have $t^{\mu}=0$, it is necessary that $\rho_\mathcal{P}^\mu=e_\mathcal{P}-1$, for any $\mathcal{P}\in \mathbb{P}_K$. By assumption, for any $L_i/L_{i-1}$ which are Kummer, not all ramified valuations of $c_i$ share a prime factor with $m_i$, whence $t^{\mu}=0$ occurs only when $\mu_i=0$ for all $i \in R_{0, \mathcal{P}}$ and $\mu_i=p-1$ for all $i \in R_{p, \mathcal{P}}$. Furthermore, for any $\mu$ so that $t^{{\mu}} = 1$, there exist no holomorphic differentials of the form prescribed in the definition of $\mathfrak{B}_L$.

By the previous argument, we have \begin{equation*}|\mathfrak{B}_L| = \sum_{{\mu} \in \Gamma} (t^{{\mu}}-1).\end{equation*} We must now show that this quantity is equal to the genus $g_L$ of $L$. By definition of $\Delta_{\mathcal{P}}^{\mu_i}$ \eqref{delta}, we have \begin{align*} \sum_{{\mu} \in {\Gamma}} \lambda_{\mathcal{P}}^{{\mu}}& =\sum_{{\mu} \in {\Gamma}} \left\lfloor \frac{\sum_{i=1}^r e(\mathfrak{P}| \mathfrak{p}_i) \Delta_{\mathfrak{p}_i}^{\mu_i}}{e_{\mathcal{P}}}\right\rfloor \\ &= \sum_{{\mu} \in {\Gamma}} \left( \left(\frac{\sum_{i=1}^r e(\mathfrak{P}| \mathfrak{p}_i) \Delta_{\mathfrak{p}_i}^{\mu_i}}{e_{\mathcal{P}}}\right) -  \left\langle\frac{\sum_{i=1}^r e(\mathfrak{P}| \mathfrak{p}_i)\Delta_{\mathfrak{p}_i}^{\mu_i}}{e_{\mathcal{P}}}\right\rangle \right).\end{align*} Via the change of index $\mu_i \rightarrow p-1 - \mu_i$ for all $i \in R_{p, \mathcal{P}}$ (see also \cite[Satz 15]{Bos}), which does not alter the value of the sum, we may write \begin{equation*}\begin{array}{lll} \sum_{{\mu} \in {\Gamma}} \lambda_{\mathcal{P}}^{{\mu}} =& \sum_{{\mu} \in {\Gamma}} \Bigg[ \left(\frac{\left[\sum_{i\in R_{0, \mathcal{P}}}e(\mathfrak{P}| \mathfrak{p}_i)\mu_{i} v_{\mathcal{P},i}-\sum_{i\in R_{p, \mathcal{P}}}e(\mathfrak{P}| \mathfrak{p}_i)\mu_{i} v_{\mathcal{P},i}\right] + e_{\mathcal{P}}-1}{e_{\mathcal{P}}}\right) \\ & -  \left\langle \frac{\left[\sum_{i\in R_{0, \mathcal{P}}}e(\mathfrak{P}| \mathfrak{p}_i)\mu_{i} v_{\mathcal{P},i}-\sum_{i\in R_{p, \mathcal{P}}}e(\mathfrak{P}| \mathfrak{p}_i)\mu_{i} v_{\mathcal{P},i}\right] + e_{\mathcal{P}}-1}{e_{\mathcal{P}}}\right\rangle \Bigg].\end{array}\end{equation*} By construction, for each $s \in S$, the set \begin{equation*}\left\{\left.\left[\sum_{i\in R_{0, \mathcal{P}}}e(\mathfrak{P}| \mathfrak{p}_i)\mu_{i} v_{\mathcal{P},i}-\sum_{i\in R_{p, \mathcal{P}}}e(\mathfrak{P}| \mathfrak{p}_i)\mu_{i} v_{\mathcal{P},i}\right] +e_{\mathcal{P}}-1\;\right|\; (\mu_{1},\ldots,\mu_{r}) \in \Gamma_{s,\mathcal{P}} \right\}\end{equation*} forms a complete system of residues modulo $e_{\mathcal{P}}$. As $|S|=\prod_{i \in R_{0, \mathcal{P}}} m_i/ e(\mathfrak{p}_i|\mathfrak{p}_{i-1})$, we therefore find that \begin{align*} \sum_{{\mu} \in {\Gamma}} \left(\lambda_{\mathcal{P}}^{{\mu}} -\sum_{i\in R_{0, \mathcal{P}}} \frac{ e ( \mathfrak{P} | \mathfrak{p}_{i})}{e_\mathcal{P}}  v_{\mathcal{P},i}\mu_i \right)&=
 \sum_{\mu \in \Gamma }\left[\frac{-\sum_{i\in R_{p, \mathcal{P}}} e(\mathfrak{P}|\mathfrak{p}_i) v_{\mathcal{P},i}\mu_{i} }{e_\mathcal{P}}\right] +\frac{[L:K]}{e_\mathcal{P}} \cdot \frac{e_\mathcal{P}-1}{2} \\ &= \frac{[L:K]}{e_\mathcal{P}} \cdot \frac{1}{2}\left(\left[-\sum_{i\in R_{p, \mathcal{P}}} e(\mathfrak{P}|\mathfrak{p}_i)v_{\mathcal{P},i}(p-1)\right]+ e_\mathcal{P} -1 \right).  \end{align*} We note that \begin{equation*}e_\mathcal{P}-1 = \sum_{i=1}^r e(\mathfrak{P}|\mathfrak{p}_i)(e(\mathfrak{p}_i|\mathfrak{p}_{i-1})-1) = \sum_{i \in R_P} e(\mathfrak{P}|\mathfrak{p}_i)(e(\mathfrak{p}_i|\mathfrak{p}_{i-1})-1).\end{equation*} By previous observations, for each $i \in R_{p,P}$, we have $J_{\mathcal{P},i} = 1-v_{\mathcal{P},i}$, and for each $i \in R_{0,\mathcal{P}}$, we have $J_{\mathcal{P},i} = 1$. Thus, via Lemma \ref{RH}, we obtain
\begin{align*} \sum_{\mu\in \Gamma}t^{\mu}&=\sum_{\mu\in \Gamma} \left( \sum_{\mathcal{P} \in \mathbb{P}_K} d_{\mathcal{P}} \left(\lambda_{\mathcal{P}}^{\mu} -  \sum_{i\in R_{0, \mathcal{P}}} \frac{ e ( \mathfrak{P} | \mathfrak{p}_{i})}{e_\mathcal{P}}  v_{\mathcal{P},i}\mu_i\right)\right)
 \\& = \sum_{ \mathcal{P} \in \mathbb{P}_{K}} d_{\mathcal{P}} \cdot  \frac{[L:K]}{e_\mathcal{P}} \cdot \frac{1}{2}\left(\left[-\sum_{i\in R_{p, \mathcal{P}}} e(\mathfrak{P}|\mathfrak{p}_i)v_{\mathcal{P},i}(p-1)\right]+ e_\mathcal{P} -1 \right)\\ &=  \frac{1}{2} \sum_{\mathcal{P}\in \mathbb{P}_K} \frac{[L:K]}{e_\mathcal{P}} \cdot d_{\mathcal{P}} \cdot \left[\sum_{i \in R_{ \mathcal{P}}} e ( \mathfrak{P} | \mathfrak{p}_{i})  \cdot (e(\mathfrak{p}_i|\mathfrak{p}_{i-1})-1) \cdot J_{\mathcal{P},i} \right]\\
 &=  \frac{1}{2} \sum_{\mathcal{P}\in \mathbb{P}_K} \frac{[L:K]}{e_\mathcal{P}} \cdot d_{\mathcal{P}} \cdot d(\mathfrak{P}|\mathcal{P}).\end{align*}
We therefore conclude that \begin{align*} |\mathfrak{B}_L| &= \sum_{{\mu} \in \Gamma} ( t^{{\mu}}-1 )  \\& =1- [L:K]+\frac{1}{2} \sum_{\mathcal{P}\in \mathbb{P}_K} \frac{[L:K]}{e_\mathcal{P}} \cdot d_{\mathcal{P}} \cdot d(\mathfrak{P}|\mathcal{P}) \\& =g_L.\end{align*} It follows that the set $\mathfrak{B}_L$ of $k$-linearly independent homomorphic differentials forms a basis for the $k$-vector space $\Omega_L$ of holomorphic differentials of $L$. 
\end{proof}

\begin{remark}\label{equivalence}
\begin{enumerate}
\item The quantities $t^\mu$ coincide with all of the Boseck invariants previously defined (see for example \cite{Bos, KaKo, RzViMa1, RzViMa2, VaMa}). In particular, this agrees precisely with the invariants found in \cite{KaKo}, which addresses cyclic automorphisms. This can be seen via the identity \begin{equation*}\left\lfloor \frac{p^t u + v}{p^t n} \right\rfloor= \bigg\lfloor \frac{u}{n} + \frac{\big\lfloor \frac{v}{p^t} \big\rfloor}{n}\bigg\rfloor.\end{equation*} 
\item Theorem \ref{basis} remains true for towers such that the index of ramification $e(\mathfrak{P}|\mathcal{P})$, differential exponent $d(\mathfrak{P}| \mathcal{P})$, and inertia degree $f(\mathfrak{P}| \mathcal{P})$ are independent of the choice of place $\mathfrak{P}$ of $L$ above a particular place $\mathcal{P}$ of $K$, for all such places $\mathcal{P}$ which ramify in $L$.
\item As also noted in Definition \ref{Boseckff}, in order to satisfy the assumption that the place at infinity is unramified in $L/K$, it is sufficient to find any rational point (i.e., degree one place) of $K$ which is unramified in $L$.
\end{enumerate} 
\end{remark}

\begin{remark} For the results of this section, it is not necessary to construct this basis in terms of an action of a generator of a cyclic group \cite{VaMa}, nor is it necessary to assume that the field of constants $k$ is algebraically closed, which is particularly important for situations when the constant field is finite. The basis $\mathfrak{B}_L$ is defined completely in terms of the ramification data and those valuations arising from global standard form. \end{remark}

In \S 5, we show that the basis of Theorem \ref{basis} is a useful construction for determining the Galois module structure of $\Omega_L$ when $L/K$ is cyclic, and in \S 7, we show that we may use this basis to describe the Galois action on $\Omega_L$ when $L/K$ is abelian. 

 \section{Standard form}
 
The difficulties presented by global standard form for towers do not particularly depend on whether a certain step in a tower is Kummer or Artin-Schreier. For example, Kummer extensions are no easier to handle than Artin-Schreier extensions in this respect; this includes even two-step towers, where each step consists of either kind of extension. In this section, we give an illustration of several special cases where a global standard form may be obtained for towers. We note that these cases include all of the bases which we have referenced in existing literature, as well as some others. We also elaborate on the various difficulties associated with finding a generator in global standard form. Currently, the only evidence that we have for the non-existence of global standard form for generators in a tower derives from the fact that such a standard form does not exist for unramified extensions. It would be very useful if other examples of non-existence could be found, and we also leave this as an open question.
 
\subsection{Global standard form}

We now turn to the problem of when one may determine whether a given function field $L$ is a global standard function field (Definition \ref{Boseckff}). There are a few questions herein: first, when one may find the element $x$ so that the place at infinity is unramified (which Boseck assumed for his preliminary constructions \cite{Bos}), and also, when it is possible, given a function field $L$, to find a rational subextension $k(x)$ such that $L/k(x)$ is a solvable tower with with generators at each step in global standard form. Here, we give some examples of composita of function fields which satisfy the these criteria. \\
 
\begin{example}[\emph{Abelian extensions with easy conversion from composites into standard form towers}]\label{easy} Let $L/K$ be an abelian extension of a rational field $K=k(x)$, where $k$ is a perfect field of characteristic $p$. For the sake of this example, we suppose for simplicity that the place at infinity in $k[x]$ is unramified in $L/K$. If $l$ denotes the field of constants of $L$, then $L/ l(x)$ is both abelian and geometric (and again $l$ is perfect). Since generators in global standard form exist over the rational field \cite{Has}, we may thus assume without loss of generality that $L/K$ is geometric, that is, $k=l$. It is known that $L/K$ is then a compositum of cyclic extensions of $k(x)$ (see for example \cite{Kar}). We suppose that $k$ contains the $m_i$th for of unity for each positive integer $m_i$ ($(m_i,p)=1$), for $i \in \{ 1, \ldots , r \}$, where \begin{equation*}\text{Gal}(L/K)= \mathbb{Z} / p^{t_1} \mathbb{Z} \times \cdots \times \mathbb{Z} / p^{t_s} \mathbb{Z}\times \mathbb{Z} / m_1 \mathbb{Z} \times \cdots \times \mathbb{Z} / m_r \mathbb{Z},\end{equation*} so that $k$ contains sufficient roots of unity as in Definition \ref{Boseckff}. Thus, $L/K$ may be written as a compositum of generalised Artin-Schreier extensions $A_i/K$, with $\text{Gal}(A_i/K) \cong \mathbb{Z} / p^{t_i} \mathbb{Z}$ for each $i \in \{ 1, \ldots s\}$ and Kummer extensions $K_i/K$ with $\text{Gal}(K_i/K) \cong \mathbb{Z} / m_i \mathbb{Z}$ for each $i \in \{ 1, \ldots r\}$. Via Hasse \cite{Has}, every of the Kummer extensions $K_i/K$ admits a generator $y_i$ so that $y_i^{m_i} = c_i$ and $y_i$ is in global standard form. As explained in the proof of Lemma \ref{kummer}, we cannot have partial ramification and equal ramification indices. Secondly, a generalised Artin-Schreier extension is a cyclic extension of degree equal to a power of the characteristic of $k$, and it was proven by Madden \cite{Madden} that such an extension may be expressed as a tower of Artin-Schreier extensions $A_i=A_{i,m_i}/ A_{i,m_i-1}/\cdots /A_{i,1}=K$ with, for each $j=1,\ldots,m_i$, some generator $y_{i,j}$  of $A_{i,j} / A_{i, j-1}$ possessing defining equation $y_{i,j}^p - y_{i,j}= c_{i,j}$ in global standard form. 

The following structure of $L/K$ is natural for immediately arriving at global standard form in a tower from the composites of cyclic extensions over $K$. We suppose that for any $i\neq j$, $i, j \in \{ 1, \ldots , s\}$, the ramified places of $K$ in $A_i$ are distinct from those in $A_j$. For any ramified place $\mathcal{P}$ of $K$ in $L$, we denote by $\mathfrak{P}$ a place of $L$ above $\mathcal{P}$ and $e_{i,\mathcal{P}}$ the index of ramification of $\mathcal{P}$ in the compositum $A_1 \cdots A_s  K_1 \cdots K_i$. For such a place $\mathcal{P}$, we suppose that \begin{equation} \label{divcondition} m_i \nmid v_{\mathcal{P}} (c_i) e_{i-1, \mathcal{P}}\;\;\;\;\;\; (i = 1, \ldots , r),\end{equation} where $v_{\mathcal{P}} (c_i)$ is the valuation of $c_i$ in $K$. We also suppose that the quantities $v_{\mathcal{P}} (c_i) e_{i-1, \mathcal{P}}$ do not share a prime factor with $m_i$ at every such ramified place $\mathcal{P}$, which, as noted in \S 3.2, is necessary for the existence of a global standard form. We denote \begin{equation*}\tilde{A}_{i,j}= A_1\cdots A_{i-1} A_{i,1}\dots A_{i,j}\text{ and }\tilde{K}_i=A_1\cdots A_s K_1 \cdots K_i.\end{equation*} As the fields $A_i$ $(i=1,\ldots,s)$ do not share any ramified places, the Artin-Schreier generators $y_{i,j}$ of $\tilde{A}_{i,j}/\tilde{A}_{i,j-1}$ are automatically in global standard form. Similarly, by \eqref{divcondition}, $y_i$ is a Kummer generator in global standard form of $\tilde{K}_i/ \tilde{K}_{i-1}$. Therefore, all the conditions necessary to apply Theorem \ref{basis} for $\Omega_L$ are satisfied for the extension $L/K$. \\

\begin{remark} In this case, the identification of the appropriate tower in global standard form is natural. Due to questions related to class numbers (see for example \cite{Ros}), Hasse's method of obtaining global standard form is unclear to us in general, as it relies heavily upon the use of the principal ideal domain property of the field of rational functions. \end{remark}

We now give an example of a compositum of two Artin-Schreier extensions, which in contrast to Example \ref{easy} share ramified places, and yet where global standard form is also possible. In the following example, places may be either fully or partially ramified. We expect that other examples may also be produced. \end{example}

\begin{example}[\emph{Elementary abelian extensions}] \label{second} For this example, we suppose that $L/K$ is the compositum of two Artin-Schreier extensions $L_1/K$ and $L_2/K$. We also suppose that the generators $y_1$ and $y_2$ of $L_1$ and $L_2$, respectively, are in global standard form, with \begin{equation*}y_1^p-y_1= a_1 +m_1 z^n,\end{equation*} and that \begin{equation*}y_2^p-y_2 = m_2 z,\end{equation*} where $a_1, z \in K$, $m_1,m_2 \in k^*$, and $(n,p)=1$. As $k$ is perfect, there exists an element $\alpha \in k^*$ so that $m_1 m_2^{-n}= \alpha^p$. We additionally suppose that $a_1$ and $z$ do not share any places with negative valuation, i.e., if for some place $\mathcal{P}$ of $K$, $v_{\mathcal{P}}(a_1)<0$, then also $v_{\mathcal{P}}(z)\geq 0$, and vice versa. Therefore, at any ramified place $\mathcal{P}$ of $K$ in $L$, $v_{\mathcal{P}} (a_1+m_1z)$ is either equal to $ v_{\mathcal{P}} (a_1)$ or $v_{\mathcal{P}} (z)$, and all of the ramified places of $K$ in $L_1$ such that $v_{\mathcal{P}} (a_1 +m_1 z)= v_{\mathcal{P}} (a_1)$ are unramified in $L_2$. For example, the simple equations \begin{equation*}y_1^p -y_1 = \frac{1}{x(x-1)} = \frac{1}{x-1} -\frac{1}{x} \text{   and   } y_2^p - y_2=\frac{1}{x}\end{equation*} are of this form.

We let $\tilde{y}_1 =y_1 - \alpha y_2^n $. As defined, $\tilde{y}_1$ is an Artin-Schreier generator of $L/L_2$ in global standard form. We have \begin{align*} \tilde{y}_1 ^p - \tilde{y}_1 &= y_1^p - y_1 -m_1m_2^{-n} (y_2 + m_2 z)^n + \alpha y_2^n \\& = y_1^p - y_1 -m_1m_2^{-n} \sum_{k=0}^n {n \choose k} (y_2 )^k(m_2 z)^{n-k} + \alpha y_2^n \\& = a_1-m_1m_2^{-n} \sum_{k=1}^n {n \choose k} (y_2)^k(m_2 z)^{n-k} + \alpha y_2^n \\& = a_1-m_1m_2^{-n} \sum_{k=1}^{n-1} {n \choose k} (y_2)^k(m_2 z)^{n-k} + (\alpha-m_1m_2^{-n} ) y_2^n. \end{align*} 
By the work of Wu and Scheidler \cite{WuSch}, we know that $L/L_2$ is ramified above any place $\mathcal{P}$ of $K$ such that $v_{\mathcal{P}} (a_1)<0$, and that $v_{\mathcal{P}} (z)<0$ unless $m_1=m_2$ and $n=1$. Clearly, at a place $\mathfrak{p}_2$ of $L_2$ unramified in $L$, we have $v_{\mathfrak{p}_2} (\tilde{y}_1)\geq 0$. At the ramified places $\mathfrak{p}_2$ of $L_2$ in $L$ above a place $\mathcal{P}$ of $K$ such that $v_{\mathcal{P}} (a_1)<0$, we have $v_{\mathfrak{p}_2} (\tilde{y}_1)=v_{\mathcal{P}} (a_i)$. It follows that $\tilde{y}_1$ is in standard form at $\mathfrak{p}_2$, as was $y_1$ at $\mathcal{P}$. At a ramified place $\mathfrak{p}_2$ of $L_2$ in $L$ above a place $\mathcal{P}$ of $K$ such that $v_{\mathcal{P}} (z)<0$, one has by the strict triangle inequality that when $\alpha \neq \alpha^p$,
\begin{align*} v_{\mathfrak{p}_2} (\tilde{y}_1)&=\min \{ v_{\mathfrak{p}_2} (a_1), v_{\mathfrak{p}_2} ( (y_2)^k(m_2 z)^{n-k} ), v_{\mathfrak{p}_2} ( y_2^n) \}\\
&= \min \{ v_{\mathfrak{p}_2} (a_1), v_{\mathcal{P}} (z)( k + p(n-k) ) \; (k \in \{ 1, \ldots , n\}), n v_{\mathcal{P}} (z) \}\\&= v_{\mathcal{P}} (z)( 1 + p(n-1) )<0 \end{align*}
and coprime to $p$. Here, we find the same valuation as in \cite[Corollary 3.10]{WuSch}. \end{example}

We conclude this section with an example of a $p$-elementary abelian extension of degree $p^n$, i.e., a Galois extension with Galois group equal to a product of $n$ copies of $\mathbb{Z}/ p \mathbb{Z}$, which is equivalent to a compositum of $n$ Artin-Schreier extensions. Suppose in particular that $L/K=k(x)$ is an elementary abelian extension of degree $p^n$. It has been proven by Garcia and Stichtenoth \cite{GarSti} that as soon as $k$ contains $F_{p^n}$, there exists a generator $y$ of $L/K$ such that $y^{p^n}-y= z\in K$. In order to obtain a Boseck basis via this generator, one would need to have the element $y$ in global standard form. By \cite[Lemma 1.2]{GarSti}, the elements $y_\mu$ ($\mu \in \mathbb{F}_{p^n}$) defined by $y_{\mu}^p-y_{\mu}=\mu z$ are precisely the Artin-Schreier generators of all of the subextensions of $L$ of degree $p$ over $K$. For any place $\mathcal{P}$ of $K$ ramified in $L$ and $\mu \neq \mu'$, we have for the respective places $\mathfrak{P}_\mu$ and $\mathfrak{P}_{\mu'}$ of $K(y_\mu)$ and $K(y_{\mu'})$ above $\mathcal{P}$ that \begin{equation}\label{valuations}v_{\mathfrak{P}_\mu}(y_{\mu}) = v_{\mathfrak{P}_{\mu'}}(y_{\mu'})= v_{\mathcal{P}}(z).\end{equation} If $z$ is in global standard form, then any places $\mathcal{P}$ of $K$ are either unramified or fully ramified \cite[Theorem 3.11]{WuSch}. In particular, it is impossible to associate a partially ramified place with standard form, even \emph{locally}. If we take the compositum of two Artin Schreier extensions of the form $y_1^p- y_1= z_1$ and $y_2^p- y_2= z_2$ with $z_1,z_2 \in K$ in global standard form and $v_{\mathcal{P}} (z_1) < v_{\mathcal{P}} (z_2)$ for some place $\mathcal{P}$ of $K$ ramified in $L$, then from the previous observation, $\mathcal{P}$ is fully ramified \cite{WuSch}. For a basis $\{\mu_1,\mu_2\}$ of $\mathbb{F}_{p^2}/\mathbb{F}_p$, we may find a generator $y$ of $K(y_1,y_2)/K$ with generating equation $y^{p^2}-y = z \in K$ by taking $y = \mu_1 y_1 + \mu_2 y_2$. However, the generator $y$ cannot be written in global standard form at $\mathcal{P}$ due to \eqref{valuations}.  

\begin{remark} It is unknown to us when a generator may be expressed globally in standard form in this case. To demonstrate, if a global standard form were possible, then by \cite{GarSti}, one would need to obtain generators $y_i$ of the Artin-Schreier extensions satisfying equations $y_i^p-y_i=z_i$ in global standard form, for all $p$-subextensions of $L/K$, so that for any $i\neq j$, there exists a $\mu_{i,j} \in \mathbb{F}^*_{p^n}$ so that $z_i= \mu_{i,j} z_j$. Particularly, the use of the generator $y$ satisfying $y^{p^n}-y=z$ seems to be restrictive for obtaining a basis or Riemann-Hurwitz formula in terms of the valuation of the generator. \end{remark} 
\subsection{Weak standard form}

In this section, we again assume that $L$ is a global standard function field (Definition \ref{Boseckff}). Given the difficulties associated with global standard form (see Examples \ref{easy} and \ref{second} for details), it is desirable to find a type of standard form which is somewhat weaker than global standard form for Boseck-type calculations. We therefore prove the existence of a weaker standard form than the global standard form: Precisely, we are able to find a generator for each $L_i/L_{i-1}$ such that, for this choice of generator, all of the ramified places of the tower are simultaneously in \emph{local} standard form. This is the only ingredient necessary for us to extend the proof of \cite[Lemma 8]{KaKo} from an algebraically closed constant field $k$ to any perfect constant field which contains the requisite roots of unity.

\begin{lemma}\label{ASnormal}
Let $K$ be any function field of characteristic $p > 0$ with perfect field of constants $k$, and let $L/K$ be an Artin-Schreier extension. Let $\{\mathfrak{p}_a\}_{a \in A}$ denote a finite set of places of $K$ unramified in $L$. \begin{enumerate}[(i)] \item There exists $\tilde{y}_0 \in L$ so that $L= K(\tilde{y}_0)$, the valuation of $\tilde{y}_0$ at each ramified place of $L/K$ is negative and coprime to $p$, and the valuation of $\tilde{y}_0$ at each place of $L$ above $\{\mathfrak{p}_a\}_{a \in A}$ is nonnegative. \item If an Artin-Schreier generator of $L/K$ is in standard form at each ramified place, then the valuation of the generator at those places is uniquely determined. Precisely, with $z^p - z = u$ in standard form for each ramified place $\mathcal{P}$ of $K$ in $L$, one has for each such $\mathcal{P}$ that \begin{equation*}v_{\mathfrak{P}}(z) = \max\{ v_{\mathcal{P}} ( u-(w^p -w)) | w \in K\},\end{equation*} where $\mathfrak{P}$ denotes the place of $L$ above $\mathcal{P}$. \end{enumerate}
\end{lemma}
\begin{proof}
We first prove the following: Let $\{\mathfrak{p}_t\}_{t \in T}$ be a finite set of places of $L$ and $\{\mathfrak{p}_s\}_{s \in S} \subset \{\mathfrak{p}_t\}_{t \in T}$. If $u, v \in K\backslash \{ 0\}$, and $v_{\mathfrak{p}_s} (u) = v_{\mathfrak{p}_s} (v)$ for all $s \in S$, then there exists some $w \in L$ so that \begin{enumerate}[(1)]\item $v_{\mathfrak{p}_t} (w) = 0$ for all $t \in T$; and \item $v_{\mathfrak{p}_s} (u - w^{p} v )> v_{\mathfrak{p}_s} ( u)$, for all $s \in S$.\end{enumerate} For each place $\mathfrak{p}_t$, we let $\mathcal{O}_{\mathfrak{p}_t}$ denote the valuation ring at $\mathfrak{p}_t$ and and $\mathfrak{p}_t$ the corresponding maximal ideal. By assumption, the residue class $\overline{(u/v)} \neq 0$ in $\mathcal{O}_{\mathfrak{p}_s}/\mathfrak{p}_s$, for each $s \in S$. As $k$ is perfect, so too is $\mathcal{O}_{\mathfrak{p}_s}/\mathfrak{p}_s$ perfect, as a finite extension of $k$. This implies that for each $s \in S$, there exists some $w_s \in \mathcal{O}_{\mathfrak{p}_s}^*$ so that \begin{equation*}\overline{w}_s^p = \overline{(u/v)} \in \mathcal{O}_{\mathfrak{p}_s}/\mathfrak{p}_s.\end{equation*} Let $\mathfrak{p}^*$ be a place of $K$ distinct from $\{\mathfrak{p}_t\}_{t \in T}$, which exists by the infinitude of places of $K$. By \cite[Theorem 5.7.10]{Vil}, there exists an element $\alpha \in K$ so that \begin{enumerate}[(i)] \item $v_{\mathfrak{p}^*}(\alpha) < 0$, \item and $v_\mathfrak{p}(\alpha) \geq 0$, for all $\mathfrak{p} \neq \mathfrak{p}^*$. \end{enumerate} Let $\mathcal{O}$ denote the integral closure of $k[\alpha]$ in $K$. By \cite[Theorem 5.7.9]{Vil}, \begin{equation*}\mathcal{O} = \bigcap_{\mathfrak{p} \neq \mathfrak{p}^*} \mathcal{O}_\mathfrak{p}.\end{equation*} As $\mathcal{O}$ is a holomorphy ring \cite[Proposition 3.2.9]{Sti}, for each place $\mathfrak{p}$ of $K$ with $\mathfrak{p} \neq \mathfrak{p}^*$, \begin{equation*}\mathcal{O}/(\mathfrak{p} \cap \mathcal{O}) \cong \mathcal{O}_\mathfrak{p}/\mathfrak{p}.\end{equation*} In particular, this holds for $\mathfrak{p} = \mathfrak{p}_t$, for any $t \in T$. We now select for any $t \in T \backslash S$ a unit $a_t \in \mathcal{O}_{\mathfrak{p}_t}^*$. By the previous arguments and the Chinese Remainder Theorem \cite[Theorem 1.3.6]{Neu}, we have \begin{equation}\label{CRT} \mathcal{O}/ \prod_{t \in T} (\mathfrak{p}_t \cap \mathcal{O}) \cong \bigoplus_{t \in T} \mathcal{O}/(\mathfrak{p}_t \cap \mathcal{O}) \cong \bigoplus_{t \in T} \mathcal{O}_{\mathfrak{p}_t}/{\mathfrak{p}_t}.\end{equation} Via the isomorphism \eqref{CRT}, we choose $w \in \mathcal{O}$ so that \begin{enumerate}\item $\overline{w} = \overline{w}_s$ for all $s \in S$, and \item $\overline{w} = \overline{a}_t$ for all $t \in T \backslash S$. \end{enumerate} First, we observe that this implies that the element $w$ is a unit in $\mathcal{O}_{\mathfrak{p}_t}$, for all $t \in T$, so that the condition (1) is automatically satisfied. Second, we have for all $s \in S$ that \begin{equation*}u/v = w_s^p = w^p \mod \mathfrak{p}_s\end{equation*} from which it follows for all $s \in S$ that $v_{\mathfrak{p}_s}(u/v - w^p) > 0$. Thus condition (2) is also satisfied.

As the extension $L/K$ is Artin-Schreier, it has a generator $y$ such that $y^p - y = r \in K$. We let $\{\mathfrak{p}_s\}_{s \in S}$ denote the union of the set of places of $K$ which ramify in $L$ so that $p \mid v_{\mathfrak{p}_s}(r)$ with the set of all places of $\{\mathfrak{p}_a\}_{a \in A}$ which satisfy $v_{\mathfrak{p}_a}(r) < 0$. Also, we let $\{\mathfrak{p}_t\}_{t \in T}$ denote the union of the set of all places of $K$ which ramify in $L/K$ with $\{\mathfrak{p}_a\}_{a \in A}$. For any of the unramified places $\mathfrak{p}_a$ and any place $\mathfrak{P}$ of $L$ above $\mathfrak{p}_a$, we have if $v_{\mathfrak{p}_a}(r) < 0$, \begin{equation*}p v_{\mathfrak{P}}(y) = v_{\mathfrak{P}}(y^p - y) = v_\mathfrak{P}(r) = v_{\mathfrak{p}_a}(r),\end{equation*} so that $p \mid v_{\mathfrak{p}_a}(r)$ whenever $v_{\mathfrak{p}_a}(r) < 0$. Via weak approximation \cite[Theorem 2.5.3]{Vil}, we may find an element $\beta \in K$ so that $v_{\mathfrak{p}_s}(\beta) = v_{\mathfrak{p}_s}(r)/p$, for all $s\in S$, and so that $v_{\mathfrak{p}_s}(\beta) = 0$, for all $t \in T \backslash S$. In particular, we have for all $s \in S$ that $v_{\mathfrak{p}_s}(\beta^p) = v_{\mathfrak{p}_s}(r)$, for all $s \in S$. By the first part of this proof, we find an element $w \in K$ so that $v_{\mathfrak{p}_t}(w) = 0$ for all $t \in T$, and so that \begin{equation*}v_{\mathfrak{p}_s}(r - w^p \beta^p ) > v_{\mathfrak{p}_s}(\beta^p),\end{equation*} for all $s \in S$. Thus $v_{\mathfrak{p}_s}(\beta w)= v_{\mathfrak{p}_s}(r)/p$ for all $s \in S$ and, for all $t \in T \backslash S$, we have $v_{\mathfrak{p}_t}(\beta w) = 0$. In particular, we find for all $s \in S$ that \begin{equation*}v_{\mathfrak{p}_s}(r - ((\beta w)^p - (\beta w))) \geq \min\{v_{\mathfrak{p}_s}(r - (\beta w)^p), v_{\mathfrak{p}_s}(\beta w)\} > v_{\mathfrak{p}_s}(r),\end{equation*} and for all $t \in T \backslash S$, we find that $v_{\mathfrak{p}_t}(r - ((\beta w)^p - (\beta w))) = v_{\mathfrak{p}_t}(r)$, as $v_{\mathfrak{p}_t}(r) < 0$. Choosing the element $y' = y + tw$ yields a new Artin-Schreier generation of $L/K$, where the valuations are strictly greater than those for the original $y$ at all places of $\{\mathfrak{p}_s\}_{s \in S}$, i.e., those places of $\{\mathfrak{p}_t\}_{t \in T}$ for which the valuation of $y$ is negative and divisible by $p$. Applying this replacement algorithm successively eventually terminates, which yields an Artin-Schreier generator $\tilde{y}_0^p - \tilde{y}_0 = \tilde{r}_0$ for which $v_\mathfrak{p}(\tilde{y}_0) < 0$ and coprime with $p$ for each place $\mathfrak{p}$ of $K$ ramified in $L$, and for which $v_{\mathfrak{p}_a}(\tilde{y}_0) \geq 0$ for all $a \in A$.

Part (ii) follows as in the proof of \cite[Lemma 3.7.7, Proposition 3.7.8]{Sti}. \end{proof}

As a corollary, we find that the valuations at the unramified places we choose may be made equal to zero in weak standard form. (This is not necessary for the proof of Proposition \ref{indecomposable}, which gives the decomposition of $\Omega_L$ in terms of indecomposable $k[G]$-modules.) This form was mentioned by Valentini and Madan in \cite{VaMa}, who gave it when the field of constants is algebraically closed; we include the analogous result for when $k$ is only assumed to be perfect.

\begin{corollary}\label{ASzeronormal}
Let $K$ be any function field of characteristic $p > 0$ with perfect field of constants $k\neq \mathbb{F}_p$, and let $L/K$ be an Artin-Schreier extension. Let $\{\mathfrak{p}_a\}_{a \in A}$ denote a finite set of places of $K$ unramified in $L$. There exists $\tilde{y} \in L$ so that $L= K(\tilde{y})$, the valuation of $\tilde{y}$ at each ramified place of $L/K$ is negative and coprime to $p$, and the valuation of $\tilde{y}$ at each place of $L$ above $\{\mathfrak{p}_a\}_{a \in A}$ is equal to zero. \end{corollary}

\begin{proof} We let $\{\mathfrak{p}_t\}_{t \in T}$ denote the union of the places of $K$ which ramify in $L$ with the set of unramified places $\{\mathfrak{p}_a\}_{a \in A}$. We assume the notation of Lemma \ref{ASnormal}, and with that notation, we let $S = A$. Furthermore, we also let $\tilde{r}_0$ and $\tilde{y}_0$ be as in Lemma \ref{ASnormal}.

The set $S$ may be partitioned into three subsets $S = S_1 \cup S_2 \cup S_3$, which are defined in the following way. The set $S_1$ denotes those places of $\{\mathfrak{p}_s\}_{s \in S}$ for which $v_{\mathfrak{p}_s}(\tilde{r}_0) > 0$. The set $S_2$ denotes those places of $\{\mathfrak{p}_s\}_{s \in S}$ for which $v_{\mathfrak{p}_s}(\tilde{r}_0) = 0$ and $\tilde{r}_0$ lies in the image of the Artin-Schreier map $x \rightarrow x^p -x$ in the residue field of $K$ at $\mathfrak{p}_s$. The set $S_3$ denotes the places of $S$ not belonging to $S_1$ or $S_2$. For each $s \in S_2$, let $x_s \in \mathcal{O}_{\mathfrak{p}_s}$ be an element so that \begin{equation*}\tilde{r}_0 \equiv \overline{x}_s^p - \overline{x}_s \mod \mathfrak{p}_s,\end{equation*} which exists for each such $s$ by definition of $S_2$. We note that the element $x_s$ is necessarily a unit. By the Chinese remainder theorem \cite[Theorem 1.3.6]{Neu}, let $\alpha$ be an element of $\mathcal{O}$ (with $\mathcal{O}$ defined as before) so that \begin{enumerate} \item $\overline{\alpha} \equiv 0 \mod \mathfrak{p}_s$, for all $s \in S_1$; \item $\overline{\alpha} \equiv \overline{x}_s \mod \mathfrak{p}_s$, for all $s \in S_2$; and \item $\overline{\alpha} \not\equiv 0 \mod \mathfrak{p}_s$, for all $s \in S_3$. \end{enumerate} As $\alpha \in \mathcal{O}$, the element $\tilde{r}_0 - (\alpha^p - \alpha)$ retains its negative and coprime to $p$ valuation at the places of $K$ which ramify in $L$. For all $s \in S_1$, \begin{equation*}v_{\mathfrak{p}_s}(\tilde{r}_0 - (\alpha^p - \alpha)) \geq \min\{v_{\mathfrak{p}_s}(\tilde{r}_0),v_{\mathfrak{p}_s}(\alpha^p),v_{\mathfrak{p}_s}(\alpha)\} > 0.\end{equation*} For all $s \in S_2$, \begin{equation*}\tilde{r}_0 - (\alpha^p - \alpha) \equiv \tilde{r}_0 - (x_s^p - x_s) \equiv 0 \mod \mathfrak{p}_s,\end{equation*} as $\overline{\alpha}^p - \overline{\alpha} \equiv \overline{x}_s^p - \overline{x}_s \text{ mod } \mathfrak{p}_s$, so that $v_{\mathfrak{p}_s}(\tilde{r}_0 - (\alpha^p - \alpha)) > 0$. For all $s \in S_3$, \begin{equation*}v_{\mathfrak{p}_s}(\tilde{r}_0 - (\alpha^p - \alpha)) = 0,\end{equation*} as $\tilde{r}_0$ does not lie in the image of the Artin-Schreier map of the residue field at $\mathfrak{p}_s$. Choosing $\gamma \in k \backslash \mathbb{F}_p$ (as $k \neq \mathbb{F}_p$) now yields that the element $\tilde{r}:= \tilde{r}_0 - (\alpha^p - \alpha) - (\gamma^p - \gamma)$ has negative and coprime to $p$ valuation at the places of $K$ which ramify in $L$. As $\gamma^p - \gamma \neq 0$ in the residue field at $\mathfrak{p}$ for any place $\mathfrak{p}$ of $K$, it follows for all $s \in S_1$ that \begin{equation*}v_{\mathfrak{p}_s}(\tilde{r}) = \min\{v_{\mathfrak{p}_s}(\tilde{r}_0 - (\alpha^p - \alpha)),v_{\mathfrak{p}_s}(\gamma^p - \gamma)\} = 0.\end{equation*} By the same reasoning, the same holds for all $s \in S_2$. Finally, for all $s \in S_3$, as $\tilde{r}_0$ does not lie in the image of the Artin-Schreier map in the residue field, it follows that \begin{equation*}\tilde{r}_0 \not\equiv (\alpha+\gamma)^p - (\alpha+\gamma) = (\alpha^p - \alpha) + (\gamma^p - \gamma) \mod \mathfrak{p}_s,\end{equation*} and hence for such $s$ that \begin{align*} v_{\mathfrak{p}_s}(\tilde{r}) &= v_{\mathfrak{p}_s}(\tilde{r}_0 - (\alpha^p - \alpha) - (\gamma^p - \gamma)) \\& = v_{\mathfrak{p}_s}(\tilde{r}_0 - ((\alpha+\gamma)^p - (\alpha+\gamma))) \\& = 0.\end{align*} Therefore, the element $\tilde{r}$, and thus the associated Artin-Schreier generator \begin{equation*}\tilde{y} = \tilde{y}_0 + (\alpha + \gamma),\end{equation*} are as desired.
\end{proof}

We also state and give a proof of the analogue of weak standard form for Kummer extensions.

\begin{lemma} \label{Knormal} Let $L/K$ be a Kummer extension of function fields of degree $n$, with constant field of positive characteristic $p > 0$ which contains all $n$th roots of unity. Let $\{\mathfrak{p}_a\}_{a \in A}$ denote a finite set of places of $K$ which are unramified in $L$. There exists $\tilde{y} \in L$ so that $L= K(\tilde{y})$, the valuation of $\tilde{y}$ at all places above those of $\{\mathfrak{p}_a\}_{a \in A}$ is equal to zero, and the valuation of $\tilde{y}$ at all places of $K$ which ramify in $L$ lies in the set $\{1,\ldots,n-1\}$. \end{lemma}

\begin{proof} As $k$ contains the $n$th roots of unity, there exists $y \in L$ so that $L = K(y)$ and $y^n = c \in K$. At each ramified place $\mathfrak{p}$ of $K$, we have \begin{equation*}v_{\mathfrak{p}}(c) = l_\mathfrak{p} + n q_\mathfrak{p}\end{equation*} for some integer $q_\mathfrak{p}$ and $l_\mathfrak{p} \in \{1,\ldots,n-1\}$, and for each $a \in A$, we have $v_{\mathfrak{p}_a}(c) = n q_{\mathfrak{p}_a}$ for some integer $q_{\mathfrak{p}_a}$ \cite[Proposition 3.7.3]{Sti}. By weak approximation, we choose an element $\alpha \in K$ so that \begin{enumerate} \item $v_{\mathfrak{p}}(\alpha) = -q_\mathfrak{p}$, for all places of $K$ which ramify in $L$; and \item $v_{\mathfrak{p}_a}(\alpha) = -q_{\mathfrak{p}_a}$, for all $a \in A$.  \end{enumerate} It follows that the element $\tilde{y} = \alpha y$, which satisfies $\tilde{y}^n = (\alpha y)^n = \alpha^n c$, is as desired. \end{proof}

We now state a special form of the strict triangle inequality for Artin-Schreier extensions.

\begin{lemma}\label{AS} Let $L/K$ be a cyclic, geometric extension of function fields of degree $p$, with constant field $k$ of characteristic $p > 0$. Let $\mathfrak{p}$ be a place of $K$. Suppose that the Artin-Schreier generator $y^p - y = r \in K$ of $L/K$ is in local standard form at $\mathfrak{p}$, i.e., that for all places $\mathfrak{P}$ of $L$ above $\mathfrak{p}$, $v_\mathfrak{P}(y) < 0$ and coprime to $p$ if $\mathfrak{p}$ is ramified in $L$, and $v_\mathfrak{P}(y) \geq 0$ if $\mathfrak{p}$ is unramified in $L$. For $a \in L$, let \begin{equation*}a = b_0 + b_1 y + \cdots + b_{p-1} y^{p-1}, \;\;\;\; b_0,b_1,\ldots,b_{p-1} \in K.\end{equation*} There exists a place $\mathfrak{P}$ of $L$ above $\mathfrak{p}$ so that $v_{\mathfrak{P}}(a) = m$, where  \begin{displaymath}
   m = \left\{
     \begin{array}{lr}
       \min_{0 \leq j < p}\{v_\mathfrak{P}(b_j y^j)\} & \text{ if } \mathfrak{p} \text{ is ramified in $L$}\\
       \min_{0 \leq j < p}\{v_\mathfrak{p}(b_j)\} & \text{ if } \mathfrak{p}  \text{ is unramified in $L$}
     \end{array}
   \right.
\end{displaymath} \end{lemma}
\begin{proof} The proof is precisely as in \cite[Lemma 2]{MadMad}.
\end{proof}

Our next stated result (Lemma \ref{min}) is the natural extension of Lemma \ref{AS} to generalised Artin-Schreier extensions. One may compare this to \cite[Lemma 2]{VaMa}; we remark that Lemma \ref{min} differs slightly from this result, as the generators of unramified steps in the tower do not appear, so that we only need to require the weak standard form given by Lemma \ref{ASnormal}. This is done in order to employ a version of \cite[Theorem 1]{VaMa} over a constant field which is only assumed to be perfect.

\begin{lemma}\label{min} Let $L/K$ be a cyclic, geometric extension of function fields of degree $p^s$, with constant field $k$ of characteristic $p > 0$ and Artin-Schreier tower $L=L_s/L_{s-1}/\cdots/L_0 = K$. Let $\mathfrak{p}$ be a place of $K$, and let $\mathfrak{p}_{i-1}$ be a place of $L_{i-1}$ above $\mathfrak{p}$. Suppose for each $i =1,\ldots,s$ and each such place $\mathfrak{p}_{i-1}$ that the Artin-Schreier generator $y_i^p - y_i = r_i \in L_{i-1}$ of $L_i/L_{i-1}$ is in standard form at $\mathfrak{p}_{i-1}$ if $\mathfrak{p}_{i-1}$ is ramified in $L_i$, or that $v_\mathfrak{\mathfrak{P}}(y) \geq 0$ for all $\mathfrak{P}$ of $L_i$ above $\mathfrak{p}_{i-1}$ if $\mathfrak{p}_{i-1}$ is unramified in $L_i$. For $a \in L$, let \begin{equation*}a = \sum_{\mu_1,\ldots,\mu_s} a_{\mu_1,\ldots,\mu_s} y_1^{\mu_1} \cdots y_s^{\mu_s},\;\;\;\;a_{\mu_1,\ldots,\mu_s} \in K.\end{equation*} Then \begin{equation*}\min_{\mathfrak{P}|\mathfrak{p}} v_\mathfrak{P}(a) = \min_{\substack{\mathfrak{P}|\mathfrak{p} \\ \mu_i \\ \mathfrak{p}_i|\mathfrak{p}_{i-1} \text{ramifies}}} v_{\mathfrak{P}}\left(a_{\mu_1,\ldots,\mu_s} \prod_{\substack{i \\ \mathfrak{p}_i|\mathfrak{p}_{i-1} \text{ramifies}}} y_i^{\mu_i}\right),\end{equation*} where the minimum is taken over all places $\mathfrak{P}$ of $L$ over $\mathfrak{p}$ and $\mu_i \in \{0,\ldots,p-1\}$, for all $i$ where $\mathfrak{p}_i|\mathfrak{p}_{i-1}$ ramifies.
\end{lemma}

\begin{proof} This follows by Lemma \ref{AS} and a simple induction argument. \end{proof}
We conclude this section with an application of weak standard form to the cyclic case. Let $K$ be any function field and $k$ again a perfect field of characteristic $p>0$, and let $L/K$ be a cyclic Galois extension of degree $p^t m$ with $(m,p)=1$. Denote by $G= G_p\times G_m$ the cyclic Galois group with generator $\sigma$ of $L/K$, its unique cyclic $p$-Sylow subgroup $G_p$ with generator $\sigma_p= \sigma^{m}$, and \begin{equation*}G_m\cong \mathbb{Z} / m \mathbb{Z}\cong G/G_p,\end{equation*} with generator $\sigma_m= \bar{\sigma}$, where $\bar{\sigma}$ denotes the image of $\sigma$ in $G/G_p$. The extension $L/ L^{G_p}$ is a generalised Artin-Schreier extension, which we may write as a tower \begin{equation*}L=A_{s}/ A_{s-1} / \cdots / A_{0}= L^{G_p},\end{equation*} where $A_{i}/ A_{i-1}$ are Artin-Schreier extensions and $L^{G_p} / K$ is a Kummer extension. We denote by $\mathbb{P}_K$ the set of places of $K$ which ramify in $L$. For a place $\mathcal{P}$ of $K$, we denote by $\mathfrak{P}$ of $L$ above $\mathcal{P}$, $\mathfrak{p}_{A_i}= \mathfrak{P}\cap A_i$, and $\mathfrak{p}_{Ku}= \mathfrak{P}\cap L^{G_p}$. f

By Lemma \ref{ASnormal}, we may choose an Artin-Schreier generator $y_{A_i}$ of $A_{i}/ A_{i-1}$ such that $y_{A_i}^p - y_{A_i} = c_{A_i}$, where for any place of $A_{i-1}$ unramified in $A_i$ above a ramified place of $K$ in $L$, \begin{equation*}v_{\mathcal{P},i}:=v_{\mathfrak{p}_{A_{i-1}}} (c_{A_i}) \geq 0,\end{equation*} and such that for any place of $A_{i-1}$ ramified in $A_i$, $v_{\mathcal{P},i} <0$ and coprime with $p$. Up to a base change of the form $y_{A_i} \rightarrow l_{A_i} y_{A_i}$ with $l_{A_i}\in \mathbb{F}_p$, we may suppose without loss of generality that $\sigma_p^{p^{i-1}} ( y_{A_i}) = y_{A_i}+1$. By Lemma \ref{Knormal}, we may then choose a Kummer generator $y_{Ku}^m = c_{Ku}$ such that, for any place $\mathcal{P}$ of $K$ unramified in $L^{G_p}$ which ramifies in $L$, $v_{\mathcal{P}} (c_{Ku}) = 0$, and such that for any place of $K$ ramified in $L^{G_p}$, $v_{\mathcal{P}} (c_{Ku}) > 0$. We denote $v_{\mathcal{P}, Ku}:= v_{\mathfrak{p}_{Ku}} (y_{Ku})$. Recall that we may identify $G_m$ with the group generated by a primitive $m$th root of unity $\xi$, so that up to a choice of a different generator of $G$, we may suppose that $\sigma_m (y_{Ku}) = \xi y_{Ku}$. As $L/ L^{G_p}$ is abelian, we may decompose it as $L/ L^{p, unr} / L^{G_p}$, where $L^{p, unr} / L^{G_p}$ is unramified of degree $s_{unr} \leq s$, and such that for any $s_{unr}+1 \leq i \leq s$, there is at least one ramified place in the extension $A_i/A_{i-1}$ (take $L^{p, unr}$ to be the fixed field of the product of the inertia group at the ramified places). 

Let 
\begin{equation*}\mu = (\mu_{A_1}, \cdots, \mu_{A_s} , \mu_{Ku} ) \in  \left( \prod_{i=1}^s \{ 0, \cdots , p-1\}\right) \times \{ 0, \cdots , m-1\},\end{equation*} and let $\lambda_{\mathcal{P}}^{\mu}$ and $\rho_{\mathcal{P}}^{\mu}$ be defined as in Theorem \ref{basis}, i.e.,
 \begin{align*} e_{\mathcal{P}} &\lambda_{\mathcal{P}}^{\mu} +\rho_{\mathcal{P}}^{\mu}=\\& \left[\sum_{i=s_{unr}}^s e(\mathfrak{P}| \mathfrak{p}_{A_i}) ( (p-1- \mu_{A_i}) (- v_{P,i}) + (p-1) )\right]+ e(\mathfrak{P}| \mathfrak{p}_{Ku}) (\mu_{K_u}  v_{\mathcal{P},K_u} +(e(\mathfrak{p}_{Ku} |\mathcal{P})-1) ) ,\;\;\;\;\;\;\end{align*} with $0\leq \rho_{\mathcal{P}}^{\mu}\leq e_{\mathcal{P}}-1$. Also as in Theorem \ref{basis}, we let \begin{equation} \label{tmu} t^{\mu} = \sum_{\mathcal{P} \in \mathbb{P}_K} d_{\mathcal{P}} \left(\lambda_{\mathcal{P}}^{\mu} -  \sum_{i\in R_{o,\mathcal{P}}} \frac{ e ( \mathfrak{P} | \mathfrak{p}_{i})}{e_{\mathcal{P}}}  v_{\mathcal{P},i}\mu_i\right).\end{equation} This is the same invariant as in \cite[Definition 4]{KaKo} (see also Remark \ref{equivalence}). We note that $d_{\mathcal{P}}=1$ in the special case that $k$ is algebraically closed.
 
As Lemma \ref{Knormal}, Lemma \ref{min}, and the results of Tamagawa \cite{Tam} are valid over a perfect field, and the standard form of Lemma \ref{ASnormal} is enough to prove \cite[Theorem 1]{VaMa}, we find that \cite[Theorem 7]{KaKo} (by the same arguments) holds over any perfect field of characteristic $p > 0$ containing the requisite roots of unity. Precisely:
\begin{proposition}\label{indecomposable}
Suppose that an extension $L$ is a global standard function field (Definition \ref{Boseckff}) such that $L/K$ is cyclic. To each value of $\mu$, we associate the $p$-adic expansion \begin{equation*}\mu_p = \mu_{A_s} + p \mu_{A_{s-1}} + \cdots + \mu_{A_1} p^{s-1},\end{equation*} so that $\mu$ is associated with a unique $(\mu_p,\mu_{Ku})$. We define the $k[G]$-indecomposable module $\Delta_{\mu}$ (see also \cite[Proposition 2]{KaKo}) to be the $\mu_p$-dimensional $k$-vector space with basis $\{ v_1, \cdots , v_{\mu_p}\}$ and Galois action given by $\sigma (v_i) = \xi^{\mu_{Ku}} v_i + v_{i+1}$ for all $1 \leq i \leq \mu_p -1$ and $\sigma ( v_{\mu_p} ) = \xi^{\mu_{Ku}} v_{\mu_p}$. We obtain the decomposition in indecomposable $k[G]$-modules \begin{equation*}\Omega_L \cong\oplus_{\mu \in \Gamma} \Delta_{\mu}^{d_{\mu}},\end{equation*} where $d_{\mu}$ denotes the number of times that the module $\Delta_{\mu}$ appears in this decomposition of $\Omega_L$.  Furthermore, let $t^{\mu}$ be defined according to \eqref{tmu} (see also Theorem \ref{basis}), and let the integers $\delta_{\mu}$ be defined as  \begin{displaymath}
   \delta_{\mu} = \left\{
     \begin{array}{lr}
       1 & \text{if } t^{\mu}=0\\
       0 & \text{otherwise.}
     \end{array}
   \right.\end{displaymath} Then we find:
\begin{enumerate}
\item If $0 \leq \mu_p < p^s- p^{s_{unr}}$, 
\begin{equation*}d_{\mu}= t^{(\mu_p-1, \mu_{Ku})} - t^{\mu}+ \delta_{(\mu_p-1, \mu_{Ku})} - \delta_{\mu}.\end{equation*}
\item If $\mu_p = p^s- p^{s_{unr}}$, 
\begin{itemize}
\item for $\mu_{Ku}\neq 0$, \begin{equation*}d_{\mu}= t^{(\mu_p-1, \mu_{Ku})} - \frac{1}{p^{s_{unr}}} t^{\mu}+ \delta_{(\mu_p-1, \mu_{Ku})};\end{equation*}
\item for $\mu_{Ku}= 0$, \begin{equation*}d_{\mu}= t^{(\mu_p-1, \mu_{Ku})} -  \delta_{(\mu_p-1, \mu_{Ku})}-1.\end{equation*}
\end{itemize}
\item If $p^s- p^{s_{unr}}< \mu_p \leq p^s$, then 
\begin{enumerate}
\item for $s_{unr} \neq 0$, 
\begin{itemize}
\item whenever $\mu_{Ku}= 0$ and $\mu_p= p^s- p^{s_{unr}}+1$, \begin{equation*}d_{\mu}=1;\end{equation*}
\item whenever $\mu_p = p^s$, \begin{equation*}d_{\mu}= \frac{1}{p^{s_{unr}}} ( g_{L^{p,unr}} - 1 + t^{\mu});\end{equation*}
\item $d_{\mu}= 0$, otherwise.
\end{itemize}
\item for $s_{unr} = 0$ and $\mu_p= p^l$, 
\begin{equation*}d_{\mu}=  g_{K} - 1 + t^{\mu} + \delta_{\mu}.\end{equation*}
\end{enumerate}
\end{enumerate}
\end{proposition} 
\begin{proof} This follows easily by \cite[Theorem 7]{KaKo}, replacing the use of \cite[Theorem 1]{VaMa} with Lemma \ref{ASnormal} if $k$ is not algebraically closed. \end{proof}
\begin{remark} The invariants in the previous theorem are the same as those in \cite[Theorem 7]{KaKo} (see Remark \ref{equivalence}), as $d_{\mathcal{P}}=1$ whenever $k$ algebraically closed . \end{remark}

\section{The Galois action on $\Omega_L$ for abelian extensions} 

We now describe the Galois action on $\Omega_L$ when $L$ is a global standard function field (Definition \ref{Boseckff}) and $L/K$ is abelian. As $G=\text{Gal}(L/K)$ is abelian, we may write $L/K$ as a tower $L/ L^{G_p}/K$, where $L / L^{G_p}$ is the maximal subextension of prime power degree $p^\tau$ and $L^{G_p}/K$ is an extension of degree $m$ with $(m,p)=1$. We denote $\text{Gal}( L/ L^{G_p})= G_p$ and $\text{Gal}( L^{G_p} / K)= G_m$. Thus, $G \cong G_p\times G_m$, with \begin{equation*}G_p := \mathbb{Z} / p^{t_1} \mathbb{Z}\times \cdots \times \mathbb{Z} / p^{t_s} \mathbb{Z}\text{ and }G_m:=\mathbb{Z} / m_1 \mathbb{Z} \times \cdots \times \mathbb{Z} / m_r \mathbb{Z}.\end{equation*}
We write $|G_p|= p^\tau$ and $|G_m|=m$. We write $L/L^{G_p}$ as a tower \begin{equation*}L/L^{G_p}=K_r /\cdots / K_0=K,\end{equation*} where each $K_i/K_{i-1}$ ($i=1,\ldots,r$) is a Kummer extension with Galois group isomorphic to  $\mathbb{Z} / {m_i} \mathbb{Z}$, and we denote by $y_i$ a Kummer generator of $K_i/K_{i-1}$ with $y_{K_i}^{m_i} = c_{K_i} \in K_{i-1}$. Given a primitive $m_i$th root of unity $\xi_i$, we know that $\text{Gal}(K_i/K_{i-1})$ may be identified with the group generated by $\xi_i$, and that the action of $\xi_i$ on $y_{K_i}$ is given by $y_{K_i}\rightarrow \xi_i y_{K_i}$. By Galois theory, the $p$-extension $L/L^{G_p}$ may be expressed as a tower 
\begin{equation*}L=A_s/ \cdots / A_0=L^{G_p},\end{equation*} 
where each $A_{i}/ A_{i-1}$ ($i=1,\ldots,s$) is a generalised Artin-Schreier extension with a unique decomposition
\begin{equation*}A_{i}=A_{i,t_i}/ \cdots / A_{i,0}=A_{i-1},\end{equation*} where each $A_{i,j}/ A_{i,j-1}$ ($j = 0, \cdots , t_{i}$) is an Artin-Schreier extension with an Artin-Schreier generator $y_{A_{i,j}}$ such that $y_{A_{i,j}}^p-y_{A_{i,j}}=c_{A_{i,j}}$. We denote by $\sigma_{A_i}$ a generator of $\text{Gal}(A_i/A_{i-1})$; by Artin-Schreier theory, we may assume that $\sigma_{A_i}^{p^{i-1}} (y_{A_{i,j}}) = y_{A_{i,j}}+1$. For each $i = 1 , \cdots, s$ and $\mu_{A_i} \in \{ 0, \cdots, p^{t_i}-1\}$, we denote the $p$-adic expansion of  $\mu_{A_i}$ by \begin{equation}\label{padic} \mu_{A_i}= \mu_{A_{i,1}} + p \mu_{A_{i,2}} + \cdots + p^{t_i-1}\mu_{A_{i,t_i}}.\end{equation} We also let \begin{equation} \label{mudef} \mu = (\mu_{A_1}, \cdots , \mu_{A_s} , \mu_{K_1}, \cdots , \mu_{K_r}) \text{ and } z_{\mu} =y_{A_1}^{\mu_{A_1}}\cdots y_{A_s}^{\mu_{A_s}}y_{K_1}^{\mu_{K_1}}\cdots y_{K_r}^{\mu_{K_r}},\end{equation} where $\mu_{K_j} \in  \{  0, \cdots, m_{j}-1\}$ for each $j = 1 , \ldots, r $ and  \begin{equation*}y_{A_i}^{\mu_{A_i}}=y_{A_{i,t_j}}^{\mu_{A_{i,t_{j}}}}\cdots y_{A_{i,1}}^{\mu_{A_{i,1}}}.\end{equation*} This notation will be employed in Lemmas \ref{first}-\ref{previouslemma} and Theorem \ref{modulestructure}.

We recall that the Galois action on the tower $L/K$ may be expressed as in the following two lemmas; the proofs are just the same as in \cite{RzViMa2}, to which we refer the reader for details.

\begin{lemma}[Proposition 1, \cite{RzViMa2}]\label{first} For any $i = 1, \cdots, s$,  $j =  1, \cdots , t_i $ and $h = 1 , \cdots , p-1$, we have \begin{equation*}\sigma_i^{h}(y_{A_{i,j}} )= y_{A_{i,j}}+ P_{A_{i,j}, h}(y_{A_{i,1}}, \cdots , y_{A_{i,j-1}}),\end{equation*} where $P_{A_{i,j}, h}  (T_1, \cdots , T_{j-1}) \in \mathbb{Z} [ T_1, \cdots , T_{j-1}]$. 
\end{lemma} 

\begin{lemma}[Proposition 4, \cite{RzViMa2}]\label{null}
For any \begin{equation*}\mu = (\mu_{A_1}, \cdots , \mu_{A_s} , \mu_{K_1}, \cdots , \mu_{K_r}) \in \prod_{i=1}^s \{0,\cdots , p^{t_i}-1\} \times \prod_{j=1}^r \{0,\cdots , m_j-1\}\end{equation*} and $i =1, \cdots , s$, we have \begin{equation*}(\sigma_i -1)^{\mu_{A_i}} (z^{\mu} )= \left(\mu_{A_i}!\right)  y_{A_1}^{\mu_{A_1}} \cdots y_{A_{i-1}}^{\mu_{A_{i-1}}}  y_{A_{i+1}}^{\mu_{A_{i+1}}}\cdots  y_{A_s}^{\mu_{A_s}}y_{K_1}^{\mu_{K_1}}\cdots y_{K_r}^{\mu_{K_r}}.\end{equation*} Furthermore, \begin{equation*} (\sigma_i -1)^{\mu_{A_i}+1} (z^{\mu} )= 0.\end{equation*}
\end{lemma} 
As $L$ is a global standard function field (Definition \ref{Boseckff}), we obtain by Theorem \ref{basis} the $k$-basis of $\Omega_L$ \begin{equation*}\mathfrak{B}_L = \left\{ x^\nu [g_{{\mu}}(x)]^{-1}  z^{\mu}\; dx \;|\; \mu \in \Gamma, \ 0 \leq \nu \leq t^{{\mu}}-2\right\},\end{equation*} where $g_{{\mu}}(x)$, $z^\mu$, $\Gamma$, and $t^{{\mu}}$ are defined as in Theorem \ref{basis}. We henceforth denote a given choice of element of $\mathfrak{B}_L$ by \begin{equation} \label{diffdef} w_{\mu,\nu  }:= x^\nu [g_{{\mu}}(x)]^{-1}  z^{\mu}\; dx.\end{equation} We first prove a lemma which describes the Galois action on $\mathfrak{B}_L$.
\begin{lemma}\label{previouslemma}
For each \begin{equation*}h = ( h_{A_1}, \cdots , h_{A_s} , h_{K_1}, \cdots , h_{K_r}) \in \prod_{i=1}^s \{0,\cdots , p^{t_i}-1\} \times \prod_{j=1}^r \{0,\cdots , m_j-1\},\end{equation*} 
let
\begin{equation*} \sigma^h := \sigma_{A_1}^{h_{A_1}} \cdots \sigma_{A_s}^{h_{A_s}}\sigma_{K_1}^{h_{K_1}} \cdots \sigma_{K_r}^{h_{K_r}}.\end{equation*}
Then for every $w_{\mu,\nu} \in \mathfrak{B}_{L}$, we have \begin{align*} \sigma^h (w_{\mu , \nu})=w_{\mu,\nu} + \sum_{\substack{\mu '\in \Gamma, \mu_{A_i}'< \mu_{A_i}, i \in \{1, \cdots , s \}\\  (\mu_{K1}, \cdots, \mu_{K_r}) =(\mu_{K_1}', \cdots, \mu_{K_r}')}}  c_{\mu , \mu' ,h }\sum_{l=0}^{t^{\mu'}- t^{\mu}} B_{\mu, \mu',l} w_{\mu ', \nu +l},\end{align*} where $c_{\mu ,\mu',h}, \ B_{\mu, \mu',l}\in K$. Furthermore, let \begin{equation*}\theta_{\mu'}^{\mu, \nu}:=    w_{\mu',\nu+t^{\mu'}- t^{\mu}}+ \sum_{l=0}^{t^{\mu'}- t^{\mu}-1} B_{\mu, \mu',l} w_{\mu ', \nu +l}.\end{equation*} Then
\begin{equation*}\sigma^h (\theta_{\mu'}^{\mu, \nu}) =\theta_{\mu'}^{\mu, \nu} + \sum_{\substack{\mu ''\in \Gamma, \mu_{A_i}''<\mu_{A_i}', i \in \{1, \cdots , s \}\\  (\mu_{K1}', \cdots, \mu_{K_r}') =(\mu_{K_1}'', \cdots, \mu_{K_r}'')}} c_{\mu', \mu'', h} \theta_{\mu''}^{\mu, \nu},\end{equation*} and $\theta_{\mu'}^{\mu, \nu} \in \Omega_L$.
\end{lemma}
\begin{proof} By definition of $\sigma_{A_i}$ and $\sigma_{K_i}$, we have \begin{align*}  \sigma^h (w_{\mu, \nu})&\\=& x^{\nu} [g_{\mu} (x)]^{-1} ( \sigma_{A_1}^{h_{A_1}} (y_{A_1,1}))^{\mu_{A_1,1}} \cdots ( \sigma_{A_1}^{p^{t_1-1}h_{A_1}}(y_{A_1,t_1}))^{\mu_{A_1,t_1}} \cdots( \sigma_{A_2}^{h_{A_2}} (y_{A_s,1}))^{\mu_{A_s,1}}  \\& \quad \cdots  (\sigma_{A_s}^{p^{t_s-1}h_{A_s}} (y_{A_s,t_s}))^{\mu_{A_s,t_s}}  \cdots(\sigma_{K_1}^{h_{K_1}} (y_{K_1}))^{\mu_{K_1}} \cdots  (\sigma_{K_r}^{h_{K_r}} (y_{K_r}))^{\mu_{K_r}} dx \\=& x^{\nu} [g_{\mu} (x)]^{-1}  (y_{A_1,1}+h_{A_1})^{\mu_{A_1,1}} \cdots (y_{A_1,t_1}+P_{A_1,t_1,h} (y_{A_1,1}, \cdots , y_{A_1,t_1-1}))^{\mu_{A_1,t_1}} \cdots \\&\quad \cdots(y_{A_s,1}+h_{A_s})^{\mu_{A_s,1}} \cdots (y_{A_s,t_s}+P_{A_s,t_s,h} (y_{A_s,1}, \cdots , y_{A_s,t_s-1}))^{\mu_{A_s,t_s}} \\
&\quad(\xi_{1}^{h_{K_1}} y_{1})^{\mu_{s+1}} \cdots (\xi_r^{h_{K_r}} y_{r})^{\mu_{s+r}} dx \\ =& x^{\nu} [g_{\mu} (x)]^{-1}  \sum_{\substack{\mu '\in \Gamma, \mu_{A_i}'\leq \mu_{A_i}, i \in \{1, \cdots , s \}\\  (\mu_{K1}, \cdots, \mu_{K_r}) =(\mu_{K_1}', \cdots, \mu_{K_r}')}} c_{\mu , \mu ',h}y_{A_1,1}^{\mu_{A_1,1}'} \cdots y_{A_s,t_s}^{\mu_{A_s, t_s}'} y_{\_1}^{\mu_{K_1}} \cdots y_{K_r}^{\mu_{K_r}} dx,
\end{align*}
where $P_{A_i,j} (T_1, \cdots , T_{j-1}) \in \mathbb{Z}[T_1, \cdots , T_{j-1}]$ and $c_{\mu,\mu, h}=1$. By the proof of \cite[Theorem 2]{RzViMa2}, one can easily verify that $\mu_{A_i}'\leq \mu_{A_i}$. We thus obtain $\lambda_{\mathcal{P}}^{\mu'} \geq \lambda_{\mathcal{P}}^{\mu}$. Let
\begin{align} \label{hmu} h_{\mu ,\mu '} (x)=\frac{g_{\mu'}(x)}{g_{\mu}(x)} = \prod_{\mathcal{P} \in \mathbb{P}_K} p_{\mathcal{P}} (x)^{\lambda_{\mathcal{P}}^{\mu'}- \lambda_{\mathcal{P}}^{\mu}} =\sum_{l=0}^{t^{\mu'}- t^{\mu}} B_{\mu ,\mu ',l} x^l. \end{align}
It follows that\begin{align*} \sigma^h (w_{\mu , \nu})= \sum_{\substack{\mu '\in \Gamma, \mu_{A_i}'\leq \mu_{A_i}, i \in \{1, \cdots , s \}\\  (\mu_{K1}, \cdots, \mu_{K_r}) =(\mu_{K_1}', \cdots, \mu_{K_r}')}} c_{\mu, \mu',h }\sum_{l=0}^{t^{\mu'}- t^{\mu}} B_{\mu, \mu',l} w_{\mu ', \nu +l}. \end{align*} Note that $B_{\mu, \mu',t^{\mu'}- t^{\mu}}=1$. We also have that \begin{equation*}\nu+l \leq t^{\mu} +t^{\mu'}- t^{\mu}-2\leq t^{\mu'}-2,\end{equation*} and as a consequence, $w_{\mu', \nu+l} \in \mathfrak{B}_L$. Furthermore, we find
 \begin{align*} \sigma^h (\theta^{\mu,\nu}_{\mu'} )&= \sum_{l=0}^{t^{\mu'}- t^{\mu}} B_{\mu ,\mu ',l} \  \sigma^h (w_{\mu', \nu +l})\\ &= \sum_{l=0}^{t^{\mu'}- t^{\mu}} B_{\mu ,\mu',l} \sum_{\substack{\mu ''\in \Gamma, \mu_{A_i}''\leq \mu_{A_i}', i \in \{1, \cdots , s \}\\  (\mu_{K1}', \cdots, \mu_{K_r}') =(\mu_{K_1}'', \cdots, \mu_{K_r}'')}}c_{\mu',\mu'' , h}\sum_{k=0}^{t^{\mu''}- t^{\mu'}} B_{\mu',\mu '',k} w_{\mu'', \nu +l+k}.\end{align*} By definition of $h_{\mu,\mu ' }$ \eqref{hmu}, for all $\mu ''\leq \mu ' \leq \mu$, we obtain \begin{align*} h_{\mu ,\mu '' } (x)=h_{\mu ,\mu '} (x)h_{\mu',\mu'' } (x)=\sum_{l=0}^{t^{\mu''}- t^{\mu}} B_{\mu,\mu'',l} x^l = \sum_{l=0}^{t^{\mu''}- t^{\mu}} \sum_{e+f=l}B_{\mu',\mu'',e}B_{\mu,\mu',f} x^l.\end{align*} Thus, \begin{align*} \sigma^h (\theta^{\mu,\nu}_{\mu'} )= \sum_{\substack{\mu ''\in \Gamma, \mu_{A_i}''\leq \mu_{A_i}', i \in \{1, \cdots , s \}\\  (\mu_{K1}', \cdots, \mu_{K_r}') =(\mu_{K_1}'', \cdots, \mu_{K_r}'')}}c_{\mu',\mu'' , h}\theta^{\mu,\nu}_{\mu''} = \theta^{\mu,\nu}_{\mu'}+ \sum_{\substack{\mu ''\in \Gamma, \mu_{A_i}''< \mu_{A_i}', i \in \{1, \cdots , s \}\\  (\mu_{K1}', \cdots, \mu_{K_r}') =(\mu_{K_1}'', \cdots, \mu_{K_r}'')}}c_{\mu',\mu'' , h}\theta^{\mu,\nu}_{\mu''}.  \end{align*}
 Finally, the fact that $\theta_{\mu'}^{\mu, \nu} \in \Omega_L$ follows immediately from \cite[Theorem 2]{RzViMa2}, which concludes the proof.
\end{proof}
\begin{theorem} \label{modulestructure} Let $L$ be a global standard function field (Definition \ref{Boseckff}) such that $L/K$ is abelian. Let $p^\tau$ be the order of the $p$-Sylow subgroup $G_p$ of $\text{Gal}(L/K)$, and let $m$ be the order of the quotient $\text{Gal}(L/K)/G_p$. We define the $p$-adic expansion \begin{equation*}\mu_A:= p^{\tau -1}\mu_{A_1,1} + \cdots +  p^{\tau - t_1} \mu_{A_1,t_1}+ \cdots +  p^{t_r - 1} \mu_{A_r,1} + \cdots + \mu_{A_r,t_r} \in \{0,\ldots,p^\tau - 2\},\end{equation*} where each $\mu_{A_i,j} \in \{0,\ldots,p-1\}$ is defined as in \eqref{padic}; we also define (see \eqref{mudef}) \begin{equation*}\mu_K = (\mu_{K_1},\ldots,\mu_{K_r}) \qquad (\mu_{K_j} \in \{0,\ldots,m_j- 1\}) \end{equation*} and identify the element $\mu$ of \eqref{mudef} with the pair  $(\mu_A,\mu_K)$. We denote this identification by \begin{equation} \label{identification} \mu \sim (\mu_A,\mu_K),\end{equation} and assume it henceforth.

Then the $k$-vector space of differentials \begin{equation*}\Delta_{\mu,\nu}=\left\langle \theta^{\mu,\nu}_{\mu'}  | \ {\mu_{A}}'\leq {\mu_{A}} \ for \ i \in \{1, \cdots , s \} \ and \ (\mu_{K_1}, \cdots, \mu_{K_r}) =(\mu_{K_1}', \cdots, \mu_{K_r}')\right\rangle\end{equation*} is a $k[G]$-indecomposable submodule of $\Omega_L$ of dimension ${\mu}_A+1$. We have $\Delta_{\mu,\nu} \cong \Delta_{\mu,\nu'}$ for any $0 \leq \nu,\nu' \leq t^\mu - 2$ (see \eqref{diffdef}), and we denote $\Delta_{\mu} := \Delta_{\mu,\nu}$. Then one has the $k[G]$-module isomorphism
\begin{equation*}\Omega_L \cong  \bigoplus_{{\mu_A}=0}^{p^\tau -1} \bigoplus_{j=1}^r \bigoplus_{\mu_{K_j}=0}^{m_j-1}  \; \Delta_{\mu}^{d_\mu},\end{equation*} where ${d_\mu}$ denotes the number of times the module $\Delta_{\mu}$ appears in this decomposition of $\Omega_L$. Furthermore,  via the identification $\mu   \sim ({\mu_A}, \mu_K)$, we have \begin{equation} \label{dimj} d_{\mu}= t^{(p^\tau-2, \mu_K)} -1, \text{ and }  d_j= t^{(j-1, \mu_K)} - t^{(j,\mu_K)} , \ j= 1, \ldots , p^\tau-2,\end{equation} where $t^\mu := t^{({\mu_A}, \mu_K)}$ is defined as in Theorem \ref{basis} (see also \eqref{tmu}).
\end{theorem}
\begin{proof}  The proof of this theorem is an adaptation of various known results combined to unify abelian (but not necessarily cyclic) $p$-extensions and coprime-to-$p$ extensions. In order to not repeat existing proofs, we outline only those parts where the arguments require a little more detail. We would like to refer the reader to \cite[Theorem 1]{RzViMa1} and \cite[Proposition 2]{KaKo} for any additional details. In this proof, we follow the basic structure of the proof of \cite[Theorem 1]{RzViMa1}. The plan of the proof is as follows. \begin{enumerate} \item The module $\Delta_{\mu,\nu}$ is independent of the choice of $0 \leq \nu \leq t^\mu - 2$ and $\Delta_{\mu}$ are $\Delta_{\mu'}$ are non-isomorphic $k[G]$-modules whenever ${\mu_A} \neq {\mu_A}'$; \item We prove that $\Omega_L$ can be expressed as a direct sum of the modules $\Delta_{\mu}$ by exhaustion of basis elements; \item We count the number of times $(d_\mu)$ that $\Delta_{\mu}$ appears as a direct summand in $\Omega_L$, which is closely related to the Boseck invariant; and \item We prove that $\Delta_\mu$ is indecomposable as a $k[G]$ module, dividing the argument into two cases: \begin{itemize} \item We first consider the case where $L/K$ is a $p$-extension, that is, $m=0$, and use a counting argument as in \cite[Theorem 1]{RzViMa1} to prove that $\Delta_\mu$ is indecomposable. \item We then consider when $L/K$ has a Kummer part ($m \neq 0$), in which case one may easily employ an argument from \cite[Proposition 2]{KaKo} in combination with \cite[Theorem 1]{RzViMa1}. \end{itemize} \end{enumerate} \emph{Step 1.} The $k$-vector space \begin{equation*}\Delta_{\mu,\nu} = \left\langle \theta^{\mu,\nu}_{\mu'}  | \ {\mu_{A}}'\leq {\mu_{A}} \text{ for } i \in \{1, \ldots , s \} \text{ and } (\mu_{K_1}, \cdots, \mu_{K_r}) =(\mu_{K_1}', \cdots, \mu_{K_r}')\right\rangle\end{equation*} is a $k[G]$-module by Lemma \ref{previouslemma} (indeed, $\mu_{A_i}' \leq \mu_{A_i}$, for all $i \in \{ 1, \cdots , s\}$ implies ${\mu_{A}}'\leq {\mu_{A}}$). We observe that by definition, $\dim_k \Delta_{\mu,\nu} ={\mu_A}+1$. With $ \theta_{{\mu_A}'}:= \theta^{\mu,\nu}_{\mu'}$, we may write $\Delta_{\mu,\nu} =\langle \theta_0 , \cdots , \theta_{{\mu_A}} \rangle $ with $G$-action (Lemma \ref{previouslemma})
\begin{equation}\label{Gaction}\sigma_h (\theta_{{\mu_A}'}) =  \sum_{\substack{\mu ''\in \Gamma, \mu_{A_i}''\leq \mu_{A_i}', i \in \{1, \cdots , s \}\\  (\mu_{K1}', \cdots, \mu_{K_r}') =(\mu_{K_1}'', \cdots, \mu_{K_r}'')}}c_{\mu',\mu'' , h} \theta_{{\mu_A}''}, \qquad \qquad ({\mu_A}' \leq {\mu_A}),\end{equation}
whence, up to $k[G]$-isomorphism, $\Delta_{\mu,\nu}$ is independent of the choice of $0 \leq \nu \leq t^\mu - 2$. We therefore write $\Delta_{\mu}:=\Delta_{\mu,\nu}$ for any such choice of $\nu$. The modules $\Delta_{\mu}$ and $\Delta_{\mu'}$ are non-isomorphic $k[G]$-modules whenever $\mu \neq \mu'$, as  \begin{itemize} \item if their $p$-parts differ, i.e., ${\mu_A} \neq {\mu_A}'$, then their dimensions are not equal: \begin{equation*} \dim_k \Delta_{\mu} ={\mu_A}+1\neq {\mu_A}'+1 =\dim_k \Delta_{\mu};\end{equation*} whereas \item if ${\mu_A} = {\mu_A}'$ and $\mu \neq \mu'$, then $\mu_K\neq {\mu_K}'$, so that again $\Delta_{\mu}$ and $\Delta_{\mu'}$ are not isomorphic as $k[G]$-modules. \end{itemize}

\noindent \emph{Step 2.} Here, we employ a maximality argument similar to that of \cite[Theorem 1]{RzViMa1}. We denote 
\begin{equation*}\Delta_{{\mu_A}}:=\left\langle \theta^{\mu,\nu}_{\mu'}  | \ {\mu_{A}}'\leq {\mu_{A}} \ for \ i \in \{1, \cdots , s \} \right\rangle.\end{equation*}
By definition, we have that
 \begin{equation*}\Delta_{{\mu_A}}= \bigoplus_{j=1}^r \bigoplus_{\mu_{K_j}=0}^{m_j-1}  \Delta_{\mu}.\end{equation*}
 We fix some ${\mu_0}_K \in \prod_{j=1}^{r} \{ 0, \cdots , m_j-1\}$, and we choose $0 \leq {\mu_0}_A \leq p^\tau -2$ to be maximal such that $\mu_0 \sim ({\mu_0}_A,{\mu_0}_K)$ (see \eqref{identification}) satisfies $t^{{\mu_0}}-2 \geq 0$. Let $\nu_0 = t^{{\mu_0}}-2$ (here, $t^{\mu_0}$ is again defined as in Theorem \ref{basis}; see also \eqref{tmu}). As \begin{equation*}\theta_{\mu'}^{\mu_0, \nu_0}=    w_{\mu',\nu_0+t^{\mu'}- t^{\mu_0}}+ \sum_{l=0}^{t^{\mu'}- t^{\mu_0}-1} B_{\mu_0, \mu',l} w_{\mu ', \nu_0 +l},\end{equation*} we have $\Omega_L = M_0\oplus  \Delta_{\mu_0}$ as $k[G]$-modules, where $M_0$ is module generated by the (non-disjoint) union
\begin{align*} M_0 =& \bigg \langle \bigcup_{j=0}^{\nu_0-1} \{ w_{\mu, j+e}\;|\; 0 \leq \mu \leq {{\mu_0}_A}, \mu_K= {\mu_0}_K, 0 \leq e \leq t^{\mu} - t^{\mu_0}\}\\ 
& \qquad \bigcup \{ w_{\mu, \nu_0 +e} \;|\; 0 \leq \mu \leq {{\mu_0}_A}, \mu_K= {\mu_0}_K, 0 \leq e < t^{\mu} - t^{\mu_0}\} \bigg \rangle \\
=& \big\langle w_{\mu , \nu } \;|\; 0 \leq \mu \leq p^\tau -2 , \mu_K= {\mu_0}_K,  0 \leq \nu \leq t^{\mu} -3  \big\rangle.
\end{align*} 
We then apply the previous argument to $M_0$ and any $\mu_K \in \prod_{j=1}^r \{ 0, \cdots , m_j-1\}$, so that we may now write 
\begin{equation*}\Omega_L =  M_0' \oplus  \Delta_{{{\mu_0}_A}}.\end{equation*} 

\noindent \emph{Step 3.} We may then repeat the procedure of Part 2 of this proof with $M_0'$ in place of $\Omega_L$. More precisely, let $0 \leq \mu_1 \leq p^\tau -2$ be maximal such that $t^{\mu_1} -3 \geq 0$, and let $\nu_1 = t^{\mu_1}-3$, so that we obtain 
\begin{equation*}M_0' =  M_1' \oplus  \Delta_{{{\mu_1}_A}}.\end{equation*} 
By recursion, we may then decompose $\Omega_L$ as 
a direct sum of the modules $\Delta_{{{\mu_i}_A}}$, each of which has dimension ${{\mu_i}_A}+1 \in \{ 1, \cdots , p^\tau -1\}$. All of the modules in this direct sum of a given dimension $j$ are of the form $U= \langle \theta_0 , \cdots , \theta_{j-1} \rangle$ with $G$-action given by \eqref{Gaction}. All such modules of a given dimension are $k[G]$-isomorphic (see also the proof of \cite[Theorem 1]{RzViMa1}). We now collect the submodules isomorphic to $\Delta_\mu$; by construction,  there are $t^{(p^\tau -2,\mu_K)} -1$ submodules of dimension $p^\tau -1$ with fixed $\mu_K$ and $t^{(j-1, \mu_K)} - t^{(j, \mu_K)}$ submodules of dimension $j$ with fixed $\mu_K$, for each $j= 1, \cdots , p^\tau -2$, as in \eqref{dimj}.\\

\noindent \emph{Step 4.} We now prove that $\Delta_{\mu}$ is an indecomposable $k[G]$-module. \begin{itemize} \item If $L/K$ is a $p$-extension (that is, $m=0$), then there is simply no Kummer component of the tower; for the same reasons as in \cite[Theorem 1]{RzViMa1} and using the same computations, one can easily prove that $\dim_k ( \Omega_L^G)= t^{(0)}-1$ and deduce the indecomposability of $\Delta_{{\mu_A}}$, which is equal to $\Delta_{\mu}$ in this case. \item If $L/K$ is not a $p$-extension ($m\neq 0$), one may first note that $\Delta_{{\mu_A}} \cong \Delta_{\mu}$ as $k[G_p]$-module. The action of $\sigma^h$ of $G$ on $\Delta_\mu$ is given in terms of the matrix $ D+N$, where $D$ is a diagonal matrix with diagonal element of the form $\xi_1^{h_{K_1} \mu_{K_1}}\cdots \xi_r^{h_{K_r} \mu_{K_r}}$ and $N$ is a nilpotent matrix. Moreover, $N+I$ is the matrix giving the action of $\sigma^{h_A}$ on $\Delta_{{\mu_A}}$ ($h_A = (h_{A_1},\ldots,h_{A_s})$, see Lemma \ref{previouslemma}), which by the previous argument is an indecomposable Jordan block. An argument similar to \cite[Proposition 2]{KaKo} then proves that $\Delta_{\mu}$ is an indecomposable $k[G]$-module.\end{itemize}
\end{proof} 
\begin{remark} 
\begin{enumerate}
\item Theorem \ref{modulestructure} holds for any compositum of cyclic Kummer and Artin-Schreier over a rational field $k(x)$ which share no ramified prime and such that the place at infinity is unramified in the compositum, but it also holds for other cases, as described in \S 5.1. In particular, this holds for any abelian extension with full ramification group of the form $\mathbb{Z}/ p \mathbb{Z}$ over the rational field, with $k$ algebraically closed.
\item One may also construct a tower with each step $L_i/L_{i-1}$ admitting a generator of the form $y_i$ in global standard form with $y_i^{p^{m_i}} - y_i=a_i$ using the techniques of Theorem \ref{basis} and the identification of \begin{equation*}\mathbb{F}_{p_{m_i}}\cong \mathbb{Z}/ p \mathbb{Z} \times \cdots \times \mathbb{Z}/ p \mathbb{Z}\end{equation*} when $\mathbb{F}_{p_{m_i}} \subset k$. This would give a generalisation of Theorems \ref{basis} and  \ref{modulestructure}.
\item If $L/K$ is an elementary abelian extension of degree $p^n$, then
\begin{equation*}c_{\mu, \mu', h} =\binom{{\mu}}{{\mu'}} h_{A_{1}}^{\mu_{A_1}} \cdots h_{A_{s}}^{\mu_{A_s}}.\end{equation*}
\item The modules $\Delta_\mu$ defined in Theorem  \ref{modulestructure} do not coincide with the orbit of $w_{\mu,\nu}$, which is equal to \begin{align*}&\left\langle \theta^{\mu,\nu}_{\mu'}  | \ \mu_{A_i}'\leq \mu_{A_i} \ for \ i \in \{1, \cdots , s \}  \text{ and }  (\mu_{K_1}, \cdots, \mu_{K_r}) =(\mu_{K_1}', \cdots, \mu_{K_r}')\right\rangle\\& \cong k[G_p] / ( ( \sigma_1- 1)^{\mu_{A_1}+1} , \cdots , ( \sigma_s- 1)^{\mu_{A_s}+1}).\end{align*} This follows from Lemma \ref{null}, which yields the $k[G]$-epimorphism \begin{equation*}k[G_p] / ( ( \sigma_1- 1)^{\mu_{A_1}+1} , \cdots , ( \sigma_s- 1)^{\mu_{A_s}+1}) \rightarrow k[G]w_{\mu, \nu},\end{equation*} and the dimensions of $k[G_p] / ( ( \sigma_1- 1)^{\mu_{A_1}+1} , \cdots , ( \sigma_s- 1)^{\mu_{A_s}+1})$ and $k[G]w_{\mu, \nu}$ are the equal, resulting in a $k[G]$-\emph{isomorphism}.
\end{enumerate} 
\end{remark} 
\section{Concluding remarks}

The following are some examples which demonstrate that there is significant nontrivial structure remaining to be understood about when Theorems \ref{basis} and \ref{modulestructure} can be applied to describe the structure of the full automorphism group of a function field over its constant field.

\begin{example}[\emph{Fermat curves}]\label{fermat} Let $k = \mathbb{F}_q$ ($q = p^h$), and let $K_n/k$ ($(n,p) = 1$) be the function field defined by the equation \begin{equation*}x^n + y^n = 1.\end{equation*} Let $\zeta_n$ denote a primitive $n$th root of unity, and we suppose that $n\mid (q-1)$. This function field admits two types of automorphisms over $\mathbb{F}_q$. The first type of automorphism sends $x \rightarrow \zeta_n^a x$ and $y \rightarrow \zeta_n^b y$ $(a,b \in \mathbb{Z}/n\mathbb{Z})$; let $R_n$ denote the subgroup consisting of maps of this type. The second type of automorphism consists of two maps: one, which we call $S$, sending $x \rightarrow -\frac{y}{x}$ and $y \rightarrow \frac{1}{x}$, and the second, which we call $T$, sending $x \rightarrow \frac{1}{x}$ and $y \rightarrow -\frac{y}{x}$. We have $S^3 = T^2 = 1$ and $T^{-1} S T = S^{-1}$. Let $H$ denote the subgroup generated by $T$ and $S$. As noted by Leopoldt \cite{Leo}, if $n-1$ is not a power of $p$, then the automorphism group of the Fermat curve is given by $G_n = R_n H$. We have $R_n \unlhd G_n$, so that the group $G_n$ is the semidirect product of $\mathbb{Z}/n\mathbb{Z} \times \mathbb{Z}/n\mathbb{Z}$ with a dihedral group of order 6. As shown by Lang \cite{Lang}, the space of holomorphic differentials of this curve is generated over $k$ by the Boseck basis (the same as that given by Boseck \cite{Bos}) consisting of elements \begin{equation*}\omega_{r,s} = x^{r-1} y^{s-n} \; dx,\end{equation*} where $r,s \geq 1$ and $r + s \leq n-1$. (This is slightly different from the basis given in Lemma \ref{kummer}.) The basis $\{\omega_{r,s}\}$ provides a representation of $R_n$ via action on the term $x^{r} y^{s}$ within $\omega_{r,s}$. The fixed field of $R_n$ is equal to $K^{R_n} = \mathbb{F}_q(x^n) = \mathbb{F}_q(y^n)$. Via the generating equation of the Kummer generator $y$ of $K_n$ over $\mathbb{F}_q(x)$, the place of $\mathbb{F}_q(x)$ at infinity for $x$ is unramified in $K_n$. Thus, the conditions of Theorem \ref{basis} are satisfied for $K_n/\mathbb{F}_q(x)$. By L\"{u}roth's theorem \cite{Ros}, the fixed field $F$ of the automorphism group of $K_n$ over $\mathbb{F}_q$ is rational. We do not know if it is possible to obtain the basis of Theorem \ref{basis} for the full extension $K_n/F$.
\end{example}

\begin{example}[\emph{Artin-Mumford curves}] For $k = \overline{\mathbb{F}}_p$, the Artin-Mumford curve $\mathcal{M}_c$ for a given $c \in k$ is defined as \begin{equation*}(x^p - x)(y^p - y) = c.\end{equation*} The group of automorphisms of this curve over $\overline{\mathbb{F}}_p$ forms a semidirect product of a direct product $C_p \times C_p$ of two cyclic groups, each of order $p$, with the dihedral group $D_{p-1}$ \cite{VaMa2}. The place at infinity of $k(x)$ is unramified in $\mathcal{M}_c$, as can be seen by examining the generating equation of the Artin-Schreier generator $y$ of $\mathcal{M}_c$ over $k(x)$. Thus, the conditions of Theorem \ref{basis} are satisfied for $\mathcal{M}_c$ over $k(x)$. As with the Fermat curve in Example \ref{fermat}, the fixed field $F$ of the automorphism group of $\mathcal{M}_c$ is rational \cite{ArKo}. We do not know if it is possible to obtain the basis of Theorem \ref{basis} for the extension $\mathcal{M}_c/F$. \end{example}

\begin{remark} For a Galois extension of function fields $K/k(x)$ with field of constants $k$ for which \begin{equation*}\text{Gal}(K/k(x)) \unlhd \text{Aut}_k(K)\end{equation*} and fixed field $F$ of $\text{Aut}_k(K)$, we may examine the Galois group of $k(x)/F$. If $k$ is algebraically closed, the possible ramification behaviours and group structures of $k(x)/F$ are described completely in \cite{VaMa2}. This allows for a description of the Boseck basis and representation of $\text{Aut}_k(K)$ in terms of the extension $K/F$ via Theorem \ref{basis} if the conditions of the Theorem are satisfied.
\end{remark}

\begin{example}[\emph{Hermitian curves}]\label{hermite} Suppose that the field of constants $k=\mathbb{F}_{q^2}$. Let $K=k(x,y)/k$ be the Hermitian curve, defined by the equation \begin{equation*}y^q + y = x^{q+1}.\end{equation*} The automorphism group of $K/k$ is large (i.e., it does not satisfy the Hurwitz bound $|\text{Aut}_k(L)|\leq 84(g-1)$) and isomorphic to $PGU(3,q)$, of order $(q^3 + 1)(q^2 -1)q^3$. As with the Fermat curve, one may define the basis of holomorphic differentials in terms of the Kummer extension $K/k(y)$. As the automorphism group of the Hermitian curve is in general not solvable, the result of Theorem \ref{basis} may not be applied in this case.
\end{example}

\begin{remark} This work raised some additional questions for function fields in characteristic $p > 0$ in cases where the constant field is not assumed to be algebraically closed, which are related to when one may construct a basis as in Theorem \ref{basis} and emanate from connections between Galois and ramification theory.  

\begin{enumerate}[(i)]
\item In what generality does a global standard form exist? (See the discussion in \S 5.)
\item When does a tower of multiple Kummer and Artin-Schreier extensions form a Galois extension? 
\item Given a non-Galois tower $L/K$, when are the indices of ramification, inertia degrees, and differential exponents independent of the choice of place of $L$ above a given place of $K$, for all places of $K$? 
\item It would be interesting to find a global standard function field (Definition \ref{Boseckff}) with a non-abelian Galois tower, particularly, a Galois tower with a non-abelian $p$-group as Galois group. (If one could describe the Galois action on generators, then one could hope to learn about representations of such groups.)
\item As appears in Definition \ref{Boseckff}, when does there exist a subgroup of the automorphism group of a function field with a rational fixed field, such that the place at infinity is unramified? (This condition is necessary for Theorem \ref{basis}.)
\item Is it possible to find an explicit basis of holomorphic differentials of a field over a tower if there is an unramified step in the tower? (See the discussion in \S 3.)
\item Is it possible to construct a Boseck basis - in a tower or otherwise - when there does not exist an unramified place of degree one? (Even for cyclic automorphism groups, this would be interesting to know; this never occurs when the constant field $k$ is algebraically closed, but could happen when $k$ is finite, including when $K$ is a rational field.)
\end{enumerate}
\end{remark}
\section*{Acknowledgements}

We would like to express our gratitude to the anonymous referees for their careful questions and comments.


\end{document}